\documentclass[12pt]{amsart}

\usepackage{amsmath, amssymb}
\usepackage{array}
\usepackage[frame,cmtip,arrow,matrix,line,graph,curve]{xy}
\usepackage{graphpap, color, paralist, pstricks}
\usepackage[mathscr]{eucal}
\usepackage[pdftex]{graphicx}
\usepackage[pdftex,colorlinks,backref=page,citecolor=blue]{hyperref}
\usepackage{tikz}
\usepackage{tikz-cd}
\usepackage{enumitem}

\newtheorem{theorem}{Theorem}[section]
\newtheorem{theoremletter}{Theorem}

\newtheorem{proposition}[theorem]{Proposition}
\newtheorem{corollary}[theorem]{Corollary}
\newtheorem{lemma}[theorem]{Lemma}

\theoremstyle{definition}
\newtheorem{definition}[theorem]{Definition}
\newtheorem{example}[theorem]{Example}

\newtheorem{question}[theorem]{Question}
\newtheorem{remark}[theorem]{Remark}

\newcommand{\CC}{\mathbb{C} }

\newcommand{\RR}{\mathbb{R} }

\newcommand{\ZZ}{\mathbb{Z} }

\newcommand{\cA}{\mathcal{A} }

\newcommand{\cM}{\mathcal{M} }

\newcommand{\cS}{\mathcal{S} }

\newcommand{\cU}{\mathcal{U} }

\newcommand{\rL}{\mathrm{L} }
\newcommand{\rM}{\mathrm{M} }

\newcommand{\MRY}{\mathrm{MRY}}
\newcommand{\LRY}{\mathrm{LRY}}

\title[Cluster algebras and skein algebras]{Cluster algebras and skein algebras for surfaces}

\author{Hiroaki Karuo}
\address{Department of Mathematics, Gakushuin University, Mejiro, Toshima-ku, Tokyo, Japan}
\email{hiroaki.karuo@gakushuin.ac.jp}

\author{Han-Bom Moon}
\address{Department of Mathematics, Fordham University, New York, NY 10023}
\email{hmoon8@fordham.edu}

\author{Helen Wong}
\address{Department of Mathematical Sciences, Claremont McKenna College,
Claremont, CA 91711}
\email{hwong@cmc.edu}

\date{\today}

\begin{document}
\maketitle

\begin{abstract}
We consider two algebras of curves associated to an oriented surface of finite type -- the cluster algebra from combinatorial algebra, and the skein algebra from quantum topology. We focus on generalizations of cluster algebras and generalizations of skein algebras that include arcs whose endpoints are marked points on the boundary or in the interior of the surface. We show that the generalizations are closely related by maps that can be explicitly defined, and we explore the structural implications, including (non-)finite generation.  We also discuss open questions about the algebraic structure of the algebras.  
\end{abstract}

\section{Introduction}\label{sec:intro}

Let $\Sigma$ be an oriented surface with a finite set of marked points, which we distinguish as boundary marked points or interior punctures.  The main goals of this paper are to review and clarify relationships between several generalizations of the cluster algebra and the skein algebra of $\Sigma$,  and to explore the structural implications of this relationship.  

The \emph{cluster algebra} $\cA(\Sigma)$ of a surface are a large class of cluster algebras that are particularly well-studied (See \cite{Wil14} for an excellent survey). Cluster algebras are commutative algebras whose generators (cluster variables) and relations (cluster seeds) satisfy mutation properties determined by certain types of combinatorial data \cite{FZ02}. In the case of the cluster algebra of a surface with boundary marked points, the generators correspond to edges of an ideal triangulation, and mutation encodes the combinatorial changes seen when a diagonal edge of the triangulation is flipped \cite{GSV05, FG06}.  The definition was motivated by the similarity between the mutation formulas for cluster algebras and the Ptolemy relation for lengths of edges seen in Teichm\"uller theory, especially as in the work of Penner \cite{Pen87, Pen12}.   In the presence of interior marked points, one can also define the cluster algebra, but one needs to combinatorially extend the definition of a triangulation and the edge flip move \cite{FST08, FT18}.  

The cluster algebra for a surface has been found to be closely related to the Kauffman bracket \emph{skein algebra} $\cS_{q}(\Sigma)$  from quantum topology.  The skein algebra was introduced in \cite{Prz91, Tur91} as a generalization of the Jones polynomial \cite{Jon85, Kau87} and appears prominently in the associated topological quantum field theory \cite{BHMV95}. The skein algebra  is generated by framed loops  in $\Sigma \times (-1, 1)$ modulo the Kauffman bracket skein relations,  and it is non-commutative except for a few small surfaces and except when $q = \pm 1$.   The skein algebra holds a special place in quantum topology  due to its connections to hyperbolic geometry and algebraic geometry.  In particular, it  is a deformation quantization of the $\mathrm{SL}_2(\CC)$-character variety of the surface \cite{Bul97, PS00, Tur91}, which contains the Teichm\"uller space of $\Sigma$.

Because of the common connection to Teichm\"uller theory, it is natural to expect that there is some compatibility between the cluster algebra $\cA(\Sigma)$ and the skein algebra $\cS_{q}(\Sigma)$. Actually, because the cluster algebra is generated by edges of an ideal triangulation, a little thought shows that the compatibility should exist with versions of the skein algebra that include arcs between marked points on $\Sigma$.  One such generalization appeared in work of Roger and Yang \cite{RY14}, who wanted to identify their skein algebra as a deformation quantization of Penner's decorated Teichm\"uller space  for a punctured surface $\Sigma$ \cite{Pen87, Pen12}.  At around the same time, Muller \cite{Mul16} defined a different skein algebra, coming from a quantum cluster algebra for  a surface of boundary marked points.  Since then, these definitions have been combined and extended to include decorations of $\pm$ (called states) on the boundary marked points \cite{Le18, BKL24}.  We aim to make a precise statement about their relationship with each other and with cluster algebras here.

We remark that similar  compatibility results have been proved previously,  with different levels of generality and different methods \cite{Mul16, MQ23, MW24}.  However, clear statements about the compatibility of the cluster algebra and skein algebra have been missing so far.    The known compatibility results have treated only special cases (such as surfaces with only boundary marked points, or only with interior punctures) and only with certain variations. A goal of this paper is to provide general statements that incorporate what is known so far, both by providing references and by explaining how to extend existing proofs.  

We begin with definitions of the various generalizations of the skein algebra that include arcs in Section \ref{sec:defskein}. Our primary interest is in the \emph{ Muller-Roger-Yang generalized skein algebra} $\cS_{q}^{\MRY}(\Sigma)$ from \cite{BKL24}. We also discuss the \emph{ L\^e-Roger-Yang generalized skein algebra} $\cS_{q}^{\LRY}(\Sigma)$, which includes states at the boundary marked points.  In Section \ref{sec:finitegeneration}, we prove a striking common property, that the skein algebras are finitely generated. In contrast, as we discuss in Section \ref{ssec:applications}, the cluster algebra $\cA(\Sigma)$ is not always finitely generated. 

\begin{theoremletter}\label{thm:finitegeneration}
The generalized skein algebras $\cS_{q}^{\MRY}(\Sigma)$ and $\cS_{q}^{\LRY}(\Sigma)$ are finitely generated. 
\end{theoremletter}
It was shown that the original skein algebra $\cS_{q}(\Sigma)$ was finitely generated in  \cite{Bul99}, and the proof was extended in for specific cases \cite{Mul16, BPKW16, MW24}. Our proof here applies to the most general case.  It immediately follows that some variations of generalized skein algebras are also finitely generated (Corollary \ref{cor:finitegenerationLRY} and \ref{cor:finitegeneration}). 

In Section \ref{sec:clusteralg}, we define the cluster algebra for surfaces with both boundary marked points and interior punctures, as well its associated \emph{upper cluster algebra} $\cU(\Sigma)$.   We prove the compatibility of the cluster algebras and skein algebras in Section \ref{sec:compatibility}.  To do so, we make two modifications.  First, observe that the set of boundary edges form a multiplicative subset, and so we may define a localization $\cS_{q}^{\MRY}(\Sigma)[\partial^{-1}]$. Then we set $q = 1$ to obtain a commutative skein algebra $\cS_{1}^{\MRY}(\Sigma)[\partial^{-1}]$, and  we show that it contains a copy of the cluster algebra.  Secondly, we define a certain subalgebra $\cS_{q}^{\square}(\Sigma)$ of $\cS_{q}^{\MRY}(\Sigma)[\partial^{-1}]$ that is generated by tagged arcs and loops; see Section \ref{sec:compatibility} for the definition.   We show that it is contained inside the upper cluster algebra $\cU(\Sigma)$. 

\begin{theoremletter}\label{thm:compatibility}
Let $\Sigma$ be a triangulable marked surface. 
\begin{enumerate}
\item There is an inclusion $\cA(\Sigma) \subset \cS_{1}^{\MRY}(\Sigma)[\partial^{-1}]$. 
\item  There is a subalgebra $\cS_{q}^{\square}(\Sigma)$ of $\cS_{q}^{\MRY}(\Sigma)[\partial^{-1}]$  (that is generated by  tagged arcs and loops) so that there are inclusions
\[
	\cA(\Sigma) \subset \cS_{1}^{\square}(\Sigma) \subset \cU(\Sigma).
\]
\end{enumerate}
\end{theoremletter}

We remark that the second part of Theorem \ref{thm:compatibility} also appears in \cite[Proposition 3.7]{MQ23}. However, at least to the authors, the extension from an unpunctured surface case \cite{Mul16} to punctured surfaces needs some verification.  This is because  there are no tagged arcs in $\cS_{q}^{\MRY}(\Sigma)$, but there are vertex classes. We give some brief outline of the argument and references to relevant literature in Section \ref{sec:compatibility}.

Due to limited scope and time, the authors decided to restrict discussion here to the generalized skein algebras $\cS_{q}^{\MRY}(\Sigma)$ and $\cS_{q}^{\LRY}(\Sigma)$, and their variations.  The research on the ordinary  skein algebra $\cS_{q}(\Sigma)$  is both long and rich.  The results for the $\cS_{q}(\Sigma)$ obviously inspired the theorems presented in this paper, as well as the open questions from Section \ref{sec:open}.   Unfortunately, a full discussion and bibliography are outside the scope of our paper.   We encourage interested readers to check the references, e.g. \cite{LY21, PBIMW24}.  

\subsection*{Acknowledgements}
We thank the organizers of the special sessions on Knots, Skein Modules, and Categorification at the Joint Mathematical Meetings and Skein Modules in Low Dimensional Topology at a AMS sectional meeting in 2024 for their invitation. 
HK was supported by JSPS KAKENHI Grant Number JP23K12976. 
HW was partially supported by DMS-2305414 from the US National Science Foundation.

\section{Skein algebras}\label{sec:defskein}

Let $\underline{\Sigma}$ be a compact oriented surface with possibly nonempty boundary $\partial \underline{\Sigma}$, and $V$ be a finite set of points of $\underline{\Sigma}$. The pair $\Sigma := (\underline{\Sigma}, V)$ is called a \emph{marked surface}. From now on, if there is no chance of confusion, a surface is always a marked surface. We set $V_{\partial} := V \cap \partial \underline{\Sigma}$ and call it by the set of \emph{boundary marked points}. The complement $V_{\circ} := V \setminus V_{\partial}$ is called the set of \emph{interior punctures}.  Depending on how to treat the arc classes meeting $\partial \underline{\Sigma}$, we use two different types of \emph{tangles}. 

\begin{definition}\label{def:tangles}
A one-dimensional compact submanifold $\alpha$ of $\Sigma \times (-1, 1)$ equipped with a framing is called a \emph{$V$-tangle} if 
\begin{enumerate}
\item $\partial \alpha \subset V \times (-1, 1)$;
\item $\mathrm{Int}\; \alpha \subset (\mathrm{Int}\; \Sigma \setminus V) \times (-1, 1)$. 
\end{enumerate}
A \emph{$\partial$-tangle} is a one-dimensional compact submanifold $\alpha \subset \Sigma \times (-1, 1)$ with a framing such that 
\begin{enumerate}
\item $\partial \alpha \subset (\partial \underline{\Sigma} \setminus V_{\partial}) \times (-1, 1)$;
\item for each boundary component $e$ of $\partial \underline{\Sigma} \setminus V_{\partial}$, all points of $\partial \alpha \cap e \times (-1, 1)$ have different heights;
\item $\mathrm{Int}\; \alpha \subset (\mathrm{Int}\; \Sigma \setminus V) \times (-1, 1)$. 
\end{enumerate}
\end{definition}

If $\alpha$ is a manifold without boundary, both a $V$-tangle and a $\partial$-tangle are disjoint unions of framed loops on $\mathrm{Int} \; \Sigma \setminus V$. 

Here, we define two generalized skein algebras. For each interior puncture $v_{i} \in V_{\circ}$, we assign a formal variable $v_{i}$. We fix the Laurent polynomial ring $\CC[v_{i}^{\pm1}]$ as the coefficient ring. It commutes with all other elements that we will introduce from now on. We also fix $q \in \CC^{*}$.

\begin{definition}[Muller-Roger-Yang skein algebra 
]\label{def:MRY} 
The \emph{Muller-Roger-Yang skein algebra} $\cS_{q}^{\MRY}(\Sigma)$ of a marked surface $\Sigma = (\underline{\Sigma}, V)$ is a $\CC[v_{i}^{\pm1}]$-algebra generated by isotopy classes of $V$-tangles in $\Sigma \times (-1, 1)$, modulo the following generalized skein relations:

\begin{align*}
&({\rm A})\begin{array}{c}\includegraphics[scale=0.15]{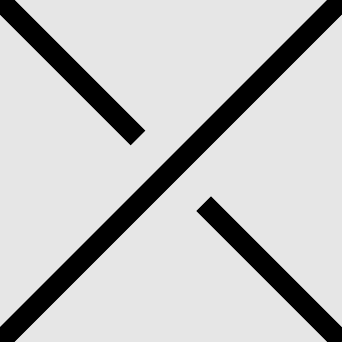}\end{array}
=q\begin{array}{c}\includegraphics[scale=0.15]{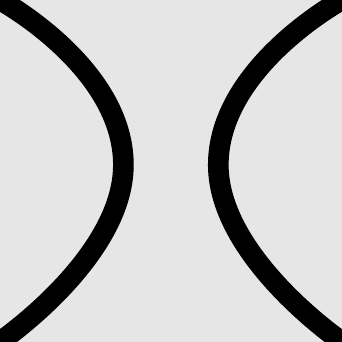}\end{array}
+q^{-1}\begin{array}{c}\includegraphics[scale=0.15]{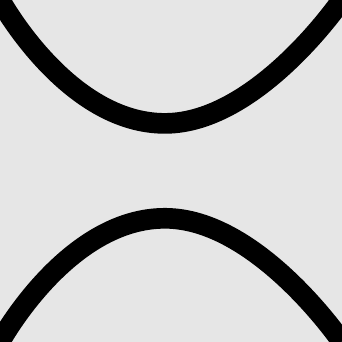}\end{array},
\\
&({\rm B})\begin{array}{c}\includegraphics[scale=0.15]{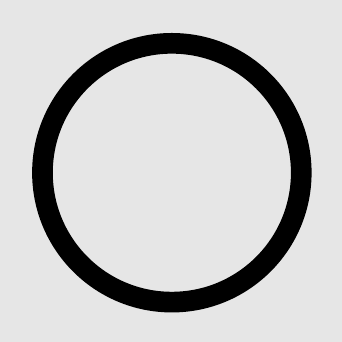}\end{array}
=(-q^2-q^{-2})\begin{array}{c}\includegraphics[scale=0.15]{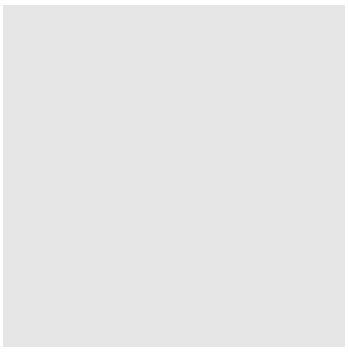}\end{array}, \\
&({\rm C}) \begin{array}{c}\includegraphics[scale=0.18]{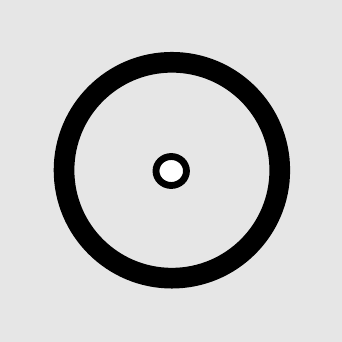}\end{array}=(q+q^{-1})\begin{array}{c}\includegraphics[scale=0.18]{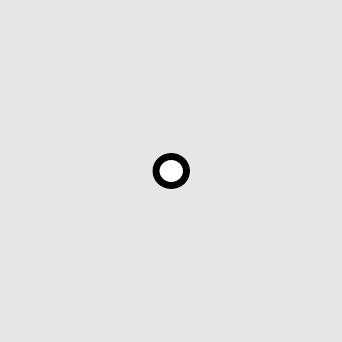}\end{array},\\
&({\rm D}) \begin{array}{c}\includegraphics[scale=0.18]{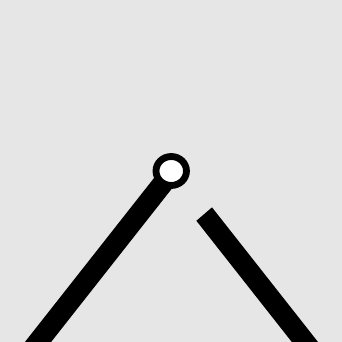}\end{array}=v^{-1}\Big(q^{1/2}\begin{array}{c}\includegraphics[scale=0.18]{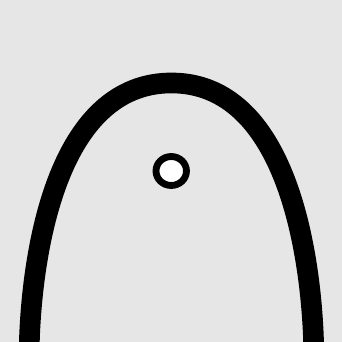}\end{array}+q^{-1/2}\begin{array}{c}\includegraphics[scale=0.18]{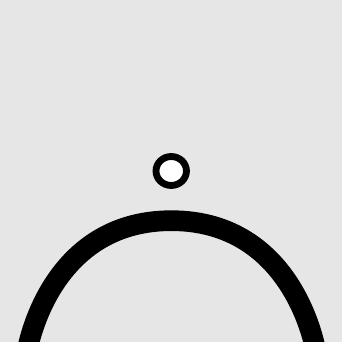}\end{array}\Big)\quad
\text{around an interior puncture $v$,}\\
&({\rm E})\ q^{-1/2}\begin{array}{c}\includegraphics[scale=0.15]{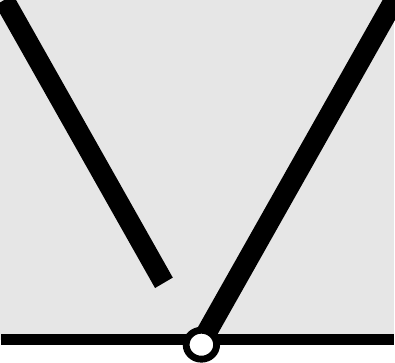}\end{array}
=q^{1/2}\begin{array}{c}\includegraphics[scale=0.15]{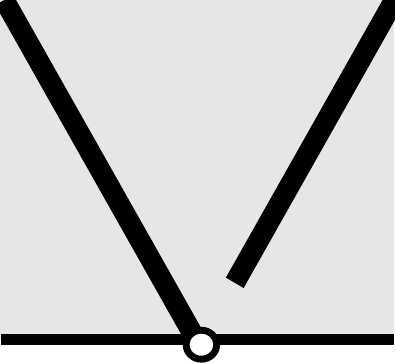}\end{array},
\\
&({\rm F})\ \begin{array}{c}\includegraphics[scale=0.15]{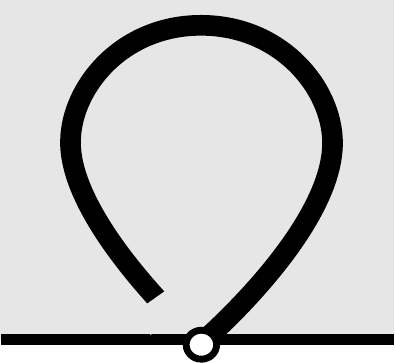}\end{array}=0=\begin{array}{c}\includegraphics[scale=0.15]{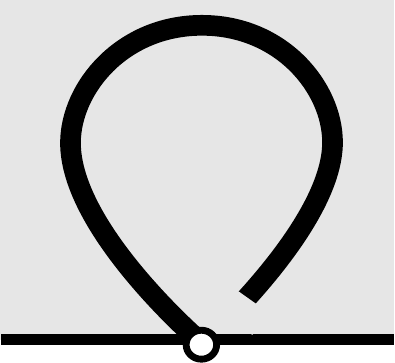}\end{array}.
\end{align*}

The multiplication is defined as the stacking of $V$-tangles. 
\end{definition}

The second generalized skein algebra is based on $\partial$-tangles and extra \emph{states}. A state of a $\partial$-tangle $\alpha$ is a map $s : \partial \alpha \cap \partial \underline{\Sigma} \to \{+, -\}$. A stated $\partial$-tangle is a pair $(\alpha, s)$. If there is no confusion, we denote it by $\alpha$. 

\begin{definition}[L\^e-Roger-Yang skein algebra ]\label{def:LRY}
The \emph{L\^e-Roger-Yang skein algebra} $\cS_{q}^{\LRY}(\Sigma)$ of a marked surface $\Sigma$ is a $\CC[v_{i}^{\pm1}]$-algebra generated by isotopy classes of stated $\partial$-tangles in $\Sigma \times (-1, 1)$, modulo the relations (A)--(D) and the following extra skein relations:

\begin{align*}
&({\rm E}') \begin{array}{c}\includegraphics[scale=0.15]{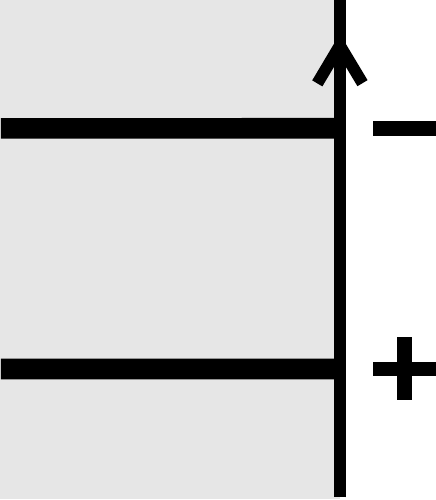}\end{array}=q^2\begin{array}{c}\includegraphics[scale=0.15]{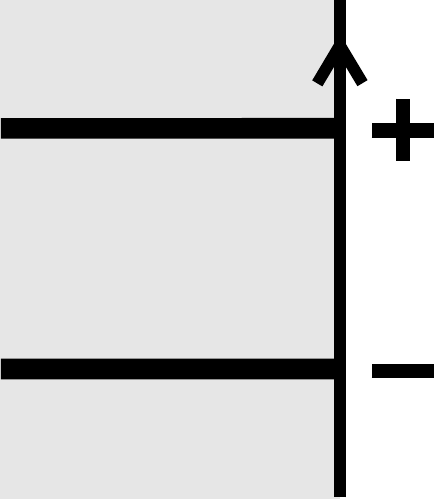}\end{array}+q^{-1/2}\begin{array}{c}\includegraphics[scale=0.15]{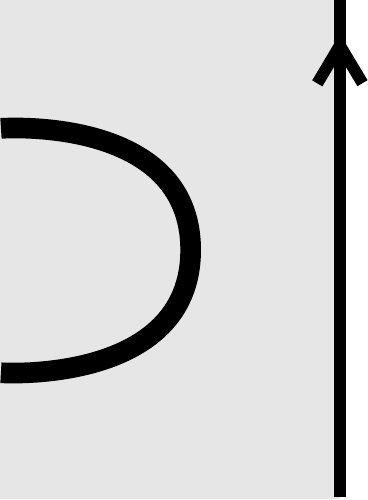}\end{array},\\
&({\rm F}')\ \begin{array}{c}\includegraphics[scale=0.15]{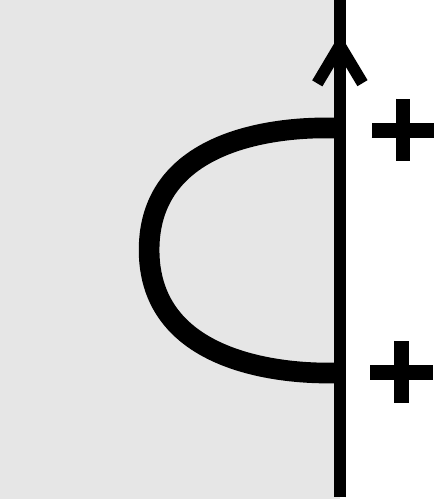}\end{array}=\begin{array}{c}\includegraphics[scale=0.15]{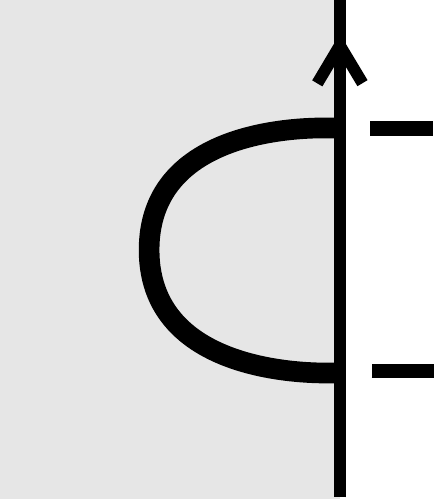}\end{array}=0 \qquad \mbox{ and } 
\qquad \begin{array}{c}\includegraphics[scale=0.15]{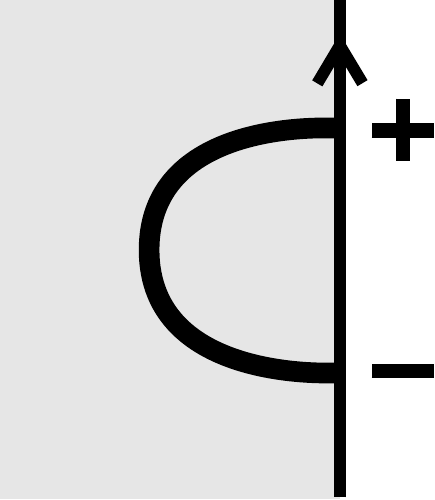}\end{array}=q^{-1/2}\begin{array}{c}\includegraphics[scale=0.15]{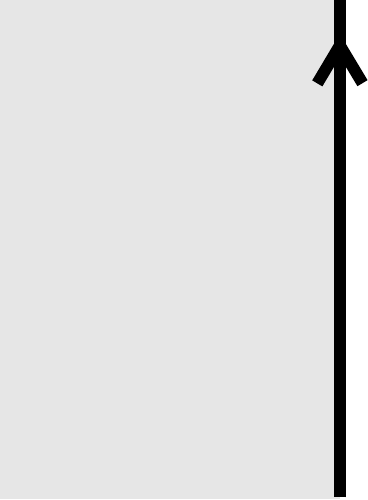}\end{array}
\end{align*}

The multiplication is defined as the stacking of $\partial$-tangles. 
\end{definition}

\begin{remark}
The ordinary Kauffman bracket skein algebra $\cS_{q}(\Sigma)$ is defined as the algebra generated by isotopy classes of loops in $\Sigma \times (-1, 1)$ modulo skein relations (A) and (B) \cite{Prz91, Tur91}. 
The definitions of $\cS_q^\MRY(\Sigma)$ and $\cS_q^\LRY(\Sigma)$ follow \cite{BKL24}, which combined generalizations of the skein algebra from recent papers including \cite{RY14, Mul16, Le18, CL22}.
If $\Sigma$ is a surface without any boundary, both $\cS_q^\MRY(\Sigma)$ and $\cS_q^\LRY(\Sigma)$ are specialized to the Roger-Yang generalized skein algebra $\cS_q^\mathrm{RY}(\Sigma)$ \cite{RY14}. If $\Sigma$ does not have any interior puncture,  $\cS_q^\MRY(\Sigma)$ is Muller's skein algebra in \cite{Mul16}, and $\cS_q^\LRY(\Sigma)$ is the stated skein algebra in \cite{Le18}. 
\end{remark}

We will also consider several variations of $\cS_q^\MRY(\Sigma)$ and $\cS_q^\LRY(\Sigma)$ that have appeared in the literature.  First, there are the subalgebras generated by tangles that avoid interior punctures.  Let $\cS_{q}^{\rM}(\Sigma)$ be the subalgebra of $\cS_{q}^{\MRY}(\Sigma)$ generated by all $V$-tangles that do not intersect any interior punctures. Similarly, $\cS_{q}^{\rL}(\Sigma)$ is the subalgebra of $\cS_{q}^{\LRY}(\Sigma)$ generated by all $\partial$-tangles not meeting interior punctures.  Note that relation (C) holds in $\cS_{q}^{\rM}(\Sigma)$ and $\cS_{q}^{\rL}(\Sigma)$. 

The next variations drop relation (C).  Let $\cS_{q}^{\rM+}(\Sigma)$ (resp. $\cS_{q}^{\rL+}(\Sigma)$) be a $\CC$-algebra generated by all $V$-tangles ($\partial$-tangles) that do not meet any interior punctures modulo the relataions (A), (B), (D), (E), and (F) (resp. (A), (B), (D), (E'), and (F')). Then we have morphisms
\begin{equation}
	\cS_{q}^{\rM+}(\Sigma) \stackrel{\pi}{\to} \cS_{q}^{\rM}(\Sigma) \stackrel{i}{\to} \cS_{q}^{\MRY}(\Sigma)
\end{equation}
and
\begin{equation}
	\cS_{q}^{\rL+}(\Sigma) \stackrel{\pi}{\to} \cS_{q}^{\rL}(\Sigma) \stackrel{i}{\to} \cS_{q}^{\LRY}(\Sigma).
\end{equation}
The first map $\pi$ is an epimorphism and $i$ is a monomorphism. 

The last variation we consider involves a quotient of the stated skein algebras.  For each boundary marked point $p$, a $\partial$-tangle in Figure \ref{fig:badarc} around $p$ is called a \emph{bad arc}. 

\begin{figure}
\includegraphics[width=0.11\textwidth]{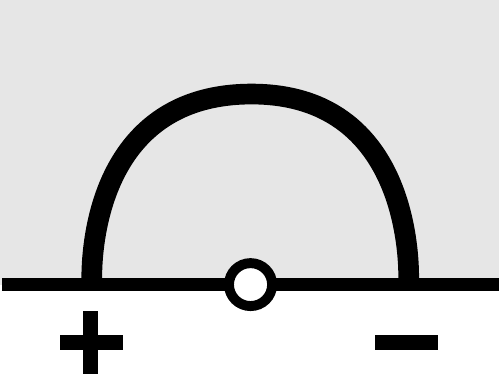}
\caption{A bad arc}
\label{fig:badarc}
\end{figure} 

\begin{definition}[\protect{\cite{CL22}}]\label{def:reducedLRY}
The \emph{reduced LRY generalized skein algebra} $\overline{\cS}_{q}^{\LRY}(\Sigma)$ is the quotient of $\cS_{q}^{\LRY}(\Sigma)$ by the two-sided ideal generated by bad arc classes. Let $p : \cS_{q}^{\LRY}(\Sigma) \to \overline{\cS}_{q}^{\LRY}(\Sigma)$ be the quotient map. 
\end{definition}

We may define similar quotients $\overline{\cS}_{q}^{\rL}(\Sigma)$ (resp. $\overline{\cS}_{q}^{\rL+}(\Sigma)$) of $\cS_{q}^{\rL}(\Sigma)$ (resp. $\cS_{q}^{\rL+}(\Sigma)$) by the bad arc classes. We will retain the same notation $p$ for the quotient map.

\section{Compatibility between skein algebras}\label{sec:compatibility}

In Section \ref{sec:defskein}, we introduced many variations of the skein algebra. From the algebraic perspective, they are not altogether `different' skein algebras, and in this section, we  will give a precise relationship between them.  Our goal will be to justify the following commutative diagram.  

\begin{equation}\label{eqn:skeinalgebradiagram}
\xymatrix{\cS_{q}^{\rM+}(\Sigma) \ar^{\pi}[r] \ar^{m}[d] &\cS_{q}^{\rM}(\Sigma) \ar^{m}[d] \ar^{i}[r] & \cS_{q}^{\MRY}(\Sigma) \ar^{m}[d]\\
\cS_{q}^{\rL+}(\Sigma) \ar^{\pi}[r] \ar^{p}[d] &\cS_{q}^{\rL}(\Sigma) \ar^{i}[r] \ar^{p}[d] & \cS_{q}^{\LRY}(\Sigma) \ar^{p}[d]\\
\overline{\cS}_{q}^{\rL+}(\Sigma) \ar^{\pi}[r] \ar@{=}^{r}[d]& \overline{\cS}_{q}^{\rL}(\Sigma) \ar^{i}[r] \ar@{=}^{r}[d] & \overline{\cS}_{q}^{\LRY}(\Sigma) \ar@{=}^{r}[d]\\
\cS_{q}^{\rM+}(\Sigma)[\partial^{-1}] \ar[r]&\cS_{q}^{\rM}(\Sigma)[\partial^{-1}] \ar[r]&\cS_{q}^{\MRY}(\Sigma)[\partial^{-1}]}
\end{equation}

From their definitions, we already have the maps $\pi$, $p$, and $i$ in the diagram. The morphisms $\pi$ and $p$ are epimorphisms, and $i$ is a monomorphism. We now define the morphisms $m$ and $r$. 

The homomorphism 
\begin{equation}\label{eqn:movinghom}
	m : \cS_{q}^{\MRY}(\Sigma) \to \cS_{q}^{\LRY}(\Sigma),
\end{equation}
comes from the `moving trick.' Recall that $\cS_{q}^{\MRY}(\Sigma)$ is generated by $V$ -tangles, which are one-dimensional submanifolds embedded in $\Sigma \times (-1, 1)$. So for each $v \in V$, the heights at the ends are all distinct. For each $v \in V_{\partial}$, we may draw a local diagram around $v$ so that all arcs ending at $v$ are not tangent to each other, as the diagram on the left in Figure \ref{fig:movingtrick}. Then, by sliding the ends to the right component of $v$, we obtain a $\partial$-tangle. For each end $\alpha \cap \partial \underline{\Sigma} \times (-1, 1)$, we assign a positive state. Using the skein relations in Definition \ref{def:MRY}, one may check that this is a well-defined map and indeed an algebra homomorphism. 

\begin{figure}[!ht]
$\begin{array}{c}\includegraphics[scale=0.18]{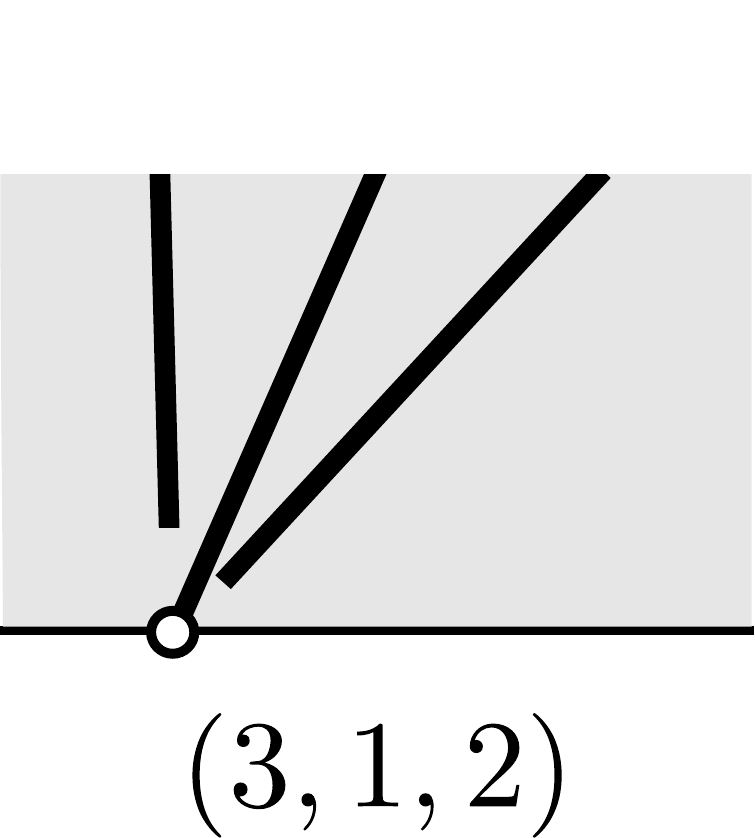}\end{array}\longleftrightarrow\begin{array}{c}\includegraphics[scale=0.18]{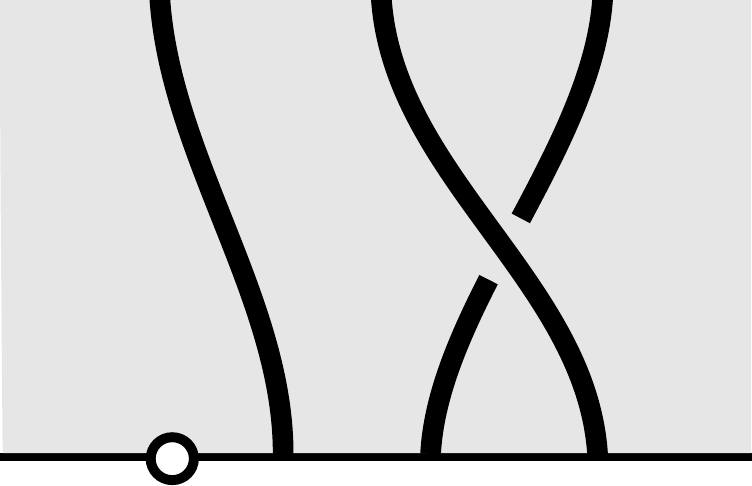}\end{array}$
\label{fig:movingtrick}
\caption{Moving trick. The numbers under the left diagram describes relative height of three ends at the boundary marked points.}
\end{figure}

The map $m$ is an identity around the interior punctures. Thus, by the same argument, we have similar homomorphisms 
\begin{equation}
	\cS_{q}^{\rM+}(\Sigma) \to \cS_{q}^{\rL+}(\Sigma), \quad 
	\cS_{q}^{\rM}(\Sigma) \to \cS_{q}^{\rL}(\Sigma). 
\end{equation}
We will use the same letter $m$ to denote the moving trick homomorphism, to avoid introducing too many notations. Because of the existence of negative states, unless $\partial \underline \Sigma = \emptyset$, $m$ is not surjective. 

We next define the map $r$.  Let $\partial$ be the multiplicative subset of $\cS_{q}^{\MRY}(\Sigma)$ generated by \emph{boundary arc classes}, which are connected $V$-tangles (hence a single arc) joining two boundary marked points on the same connected component $D$ of $\partial \Sigma$ and isotopic to the boundary edge connecting two adjacent marked points on $D$. Since any boundary arc class is $q$-commutative with any $V$-tangles, $\partial$ satisfies the Ore condition \cite[Section 2.1]{MR01}, hence we may define the Ore localization $\cS_{q}^{\MRY}(\Sigma)[\partial^{-1}]$.

\begin{proposition}\label{prop:MRYvsLRY}
For any surface $\Sigma$, there is an isomorphism 
\begin{equation}
	r : \cS_{q}^{\MRY}(\Sigma)[\partial^{-1}] \to \overline{\cS}_{q}^{\LRY}(\Sigma)
\end{equation}
from the localized MRY skein algebra to the reduced LRY skein algebra. 
\end{proposition}

\begin{proof}
We already have a homomorphism 
\begin{equation}
	r = p \circ m : \cS_{q}^{\MRY}(\Sigma) \to \cS_{q}^{\LRY}(\Sigma) \to \overline{\cS}_{q}^{\LRY}(\Sigma). 
\end{equation}
For the algebras appear here, we have explicit $\CC[v_i^{\pm1}]$-module bases, consisting of tangles without any intersection \cite[Theorem 3.6, Theorem 3.11, Proposition 6.2]{BKL24}. By checking that the basis of $\cS_q^\MRY (\Sigma)$ maps to a subset of the basis of $\overline{\cS}_q^\LRY (\Sigma)$, we may conclude that $r$ is a monomorphism. 

For any boundary arc class $\alpha \in \cS_{q}^{\MRY}(\Sigma)$ and its image $r(\alpha) \in \overline{\cS}_{q}^{\LRY}(\Sigma)$, the same curve with negative states $\overline{r(\alpha)}$ is the multiplicative inverse, so $r(\alpha)\overline{r(\alpha)}=1=\overline{r(\alpha)}r(\alpha)$ \cite[Proposition 7.4]{CL22}. By the universal property of the localization, we obtain a localized homomorphism $r : \cS_{q}^{\MRY}(\Sigma)[\partial^{-1}] \to \overline{\cS}_{q}^{\LRY}(\Sigma)$. Since the original map $r$ is injective, so is the localized $r$. 

To show the surjectivity of $r$, it is sufficient to show that any arc $\alpha$ with negative state is in $\mathrm{im}\; r$. Deforming $\alpha$ near a boundary marked point and applying the relation (E') in Definition \ref{def:LRY}, we obtain the following relation. 

\begin{align*}
&\begin{array}{c}\includegraphics[scale=0.18]{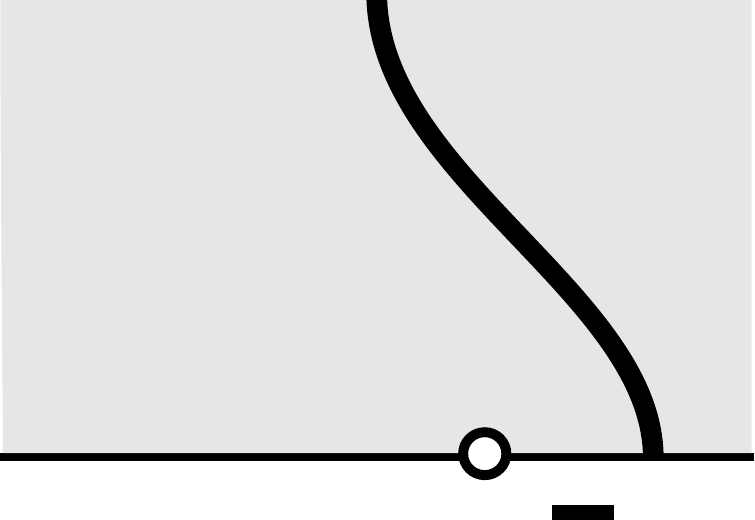}\end{array}=q^{1/2}\begin{array}{c}\includegraphics[scale=0.18]{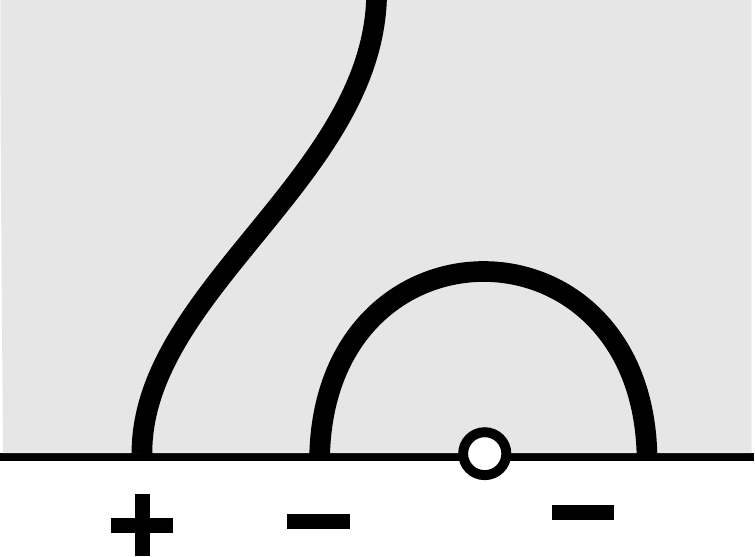}\end{array}-q^{5/2}\begin{array}{c}\includegraphics[scale=0.18]{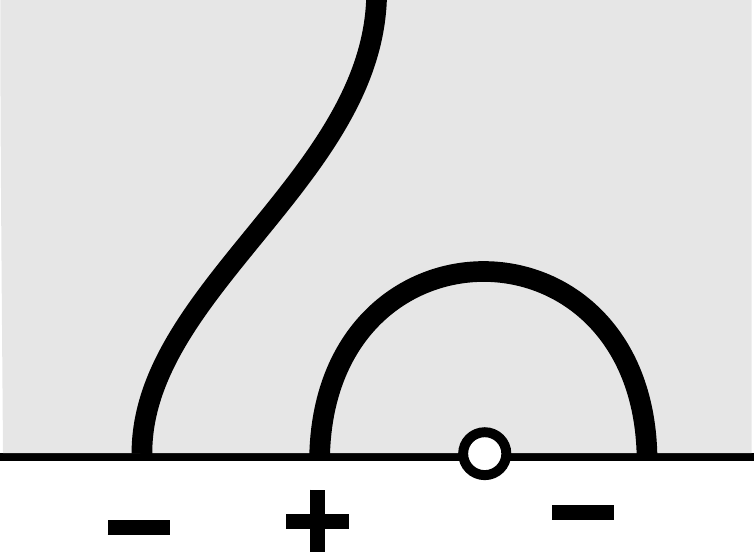}\end{array}=q^{1/2}\begin{array}{c}\includegraphics[scale=0.18]{drawing/sticking_pmm.pdf}\end{array},
\end{align*}

The first equality is from the skein relation and the second one is from the bad arc relation. The last diagram is in $\mathrm{im}\; r$. 
\end{proof}

The proof of Proposition \ref{prop:MRYvsLRY} does not affect the area near the interior punctures. Thus, the same proof works for $\cS_{q}^{\rM+}(\Sigma)$ and $\cS_{q}^{\rM}(\Sigma)$ as well. In summary, we have the equalities in Equation \eqref{eqn:skeinalgebradiagram}. 

\section{Finite generation of skein algebra}\label{sec:finitegeneration}

In this section, we prove that the skein algebras we consider are always finitely generated. In a way, this is rather surprising--- as we will see later, the skein algebra contains a copy of a cluster algebra, but in many cases the cluster algebra is not finitely generated. 
Here we give a skein theoretic proof, extending \cite{Bul99} and \cite{BKWP16}, where the finite generation of $\cS_q(\Sigma)$ and $\cS_q^\mathrm{RY}(\Sigma)$ were shown. Alternatively, one may prove the finite generation via quantum trace, as done in \cite{BKL24} for $\overline{\cS}_{q}^{\LRY}$. 

Our strategy is based on the following observation. For a marked surface $\Sigma$, the interior punctures in $V_{\circ}$ can be drawn on the boundary of a small circular disk $D$ on $\mathrm{int}\;\Sigma$ (Figure \ref{fig:surfacewithD}). Then we may understand $\Sigma$ as a union of a small open neighborhood $\widetilde{D}$ of $D$ and the outside of $D$, that is, $\Sigma \setminus \mathrm{int}\; D$. The latter is a surface without interior puncture -- now all marked points are on the boundary. We show that $\cS_{q}^{\MRY}(\Sigma)$ is generated by the arc classes in $D$ and the curve classes whose supports are in $\Sigma \setminus \mathrm{int}\; D$. The former has an explicit finite set of generators \cite[Section 5]{ACDHM21}. For the latter, we may employ an extension of the algorithm in \cite{Bul99}. 

Our discussion will be focused on simple curves, which are curves without any intersections in its diagram.  Skein relations (A) and (D) imply that the $\CC[v_i^{\pm1}]$-algebra $\cS_q^\MRY (\Sigma)$ is generated by simple curves. 

For a simple curve $\alpha$, let $c(\alpha)$ be the number of connected components of $\mathrm{int}\; D \setminus \alpha$. We may extend the definition to curves in the isotopy class of $\alpha$, by taking $c(\alpha)$ as the minimum of the possible numbers of connected components. The \emph{geometric intersection number} of $\alpha$ and $D$ is $i(\alpha) := c(\alpha) - 1$, and $i(\alpha) = 0$ if $\alpha$ does not meet $\mathrm{int}\; D$. 

\begin{figure}
\includegraphics[height=0.12\textheight]{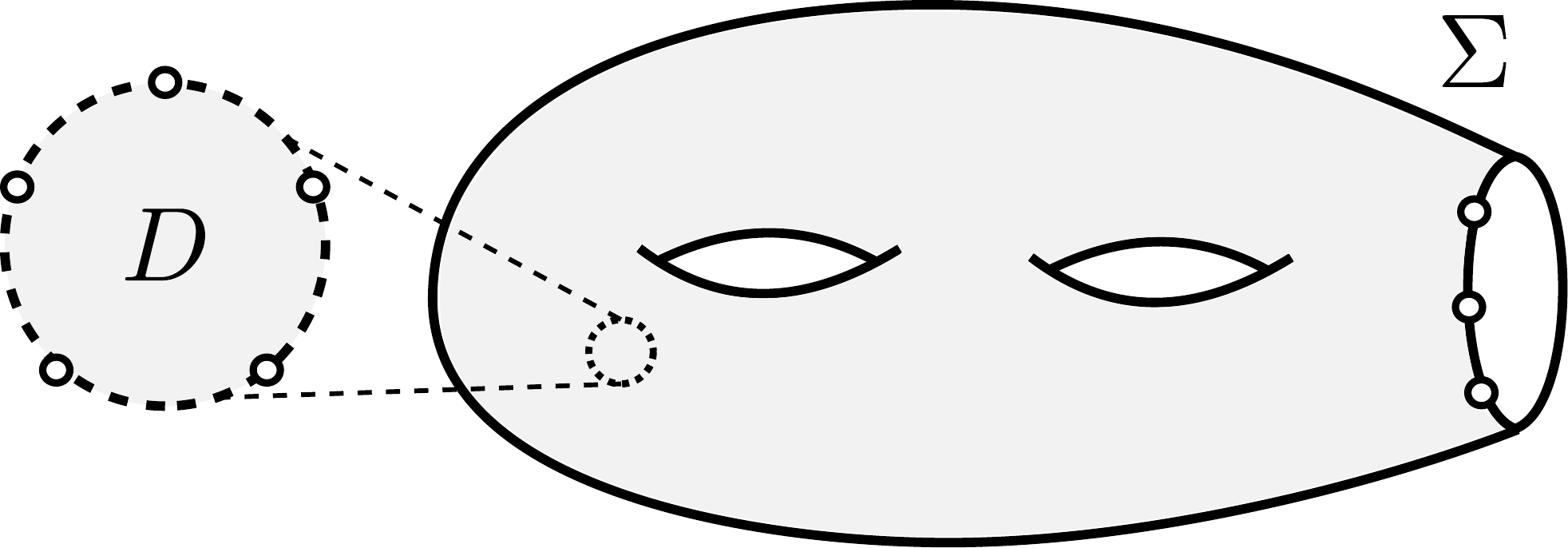}
\caption{A surface $\Sigma$ with a disk $D$ with all interior punctures on it}
\label{fig:surfacewithD}
\end{figure}

We may assume that $V_\circ\subset \partial D$ are arranged counterclockwise with respect to the center of $D$. Let $\beta_{ij}$ be the straight line segment connecting $v_{i}$ and $v_{j}$. 

\begin{lemma}\label{lem:pushout}
Let $\alpha$ be a simple curve that intersects $\mathrm{int}\; D$, so $i(\alpha) > 0$. Then $\alpha$ is generated by $\CC[v_i^{\pm1}]$-subalgebra generated by skeins which do not intersect $\mathrm{int}\; D$ and $\{\beta_{ij}\}$.
\end{lemma}

\begin{proof}
Suppose that $\alpha$ has one component with nontrivial intersections with $\mathrm{int}\; D$. Let $B$ be one of the connected components of $\alpha \cap \mathrm{int}\; D$ such that one of two components of $(\mathrm{int}\; D)\setminus B$ does not have any other components of $\alpha$. 

There are two possibilities. If one endpoint of $B$ is one of the interior punctures, then as in Figure \ref{fig:intnumreduction1}, $\alpha$ can be described as a combination of two simple curves $\alpha_{1}, \alpha_{2}$ and one $\beta_{ij}$. Note that $\alpha_{1}$ and $\alpha_{2}$ either have smaller intersection number $i$, or the component of $(\mathrm{int}\; D)\setminus B$ without other components of $\alpha$ has strictly smaller number $k$ of interior punctures. When $k = 0$, the intersection number strictly decreases, too. Thus, after finitely many steps, we may describe $\alpha$ in terms of curves with $i = 0$, and $\beta_{ij}$.

If none of the endpoints of $B$ is not an interior puncture, applying one puncture-skein relation, we may describe $\alpha$ as a curve in the previous paragraph and another curve $\alpha'$ with smaller number $k$ of interior punctures on one of the components in $(\mathrm{int}\; D) \setminus B$ (Figure \ref{fig:intnumreduction2}). Therefore, after applying the relation finite times, we obtain the conclusion. 
\end{proof}

\begin{figure}
\includegraphics[height=0.11\textheight]{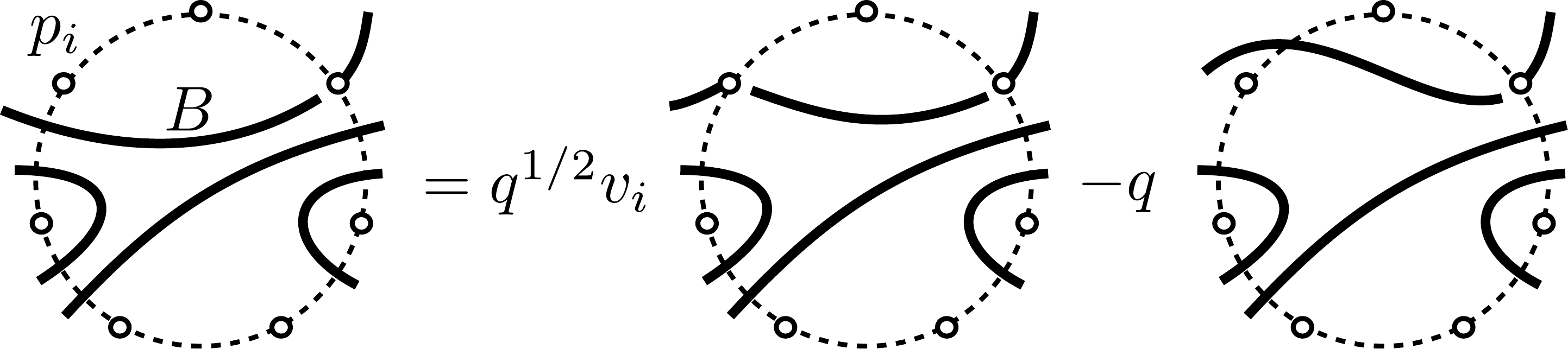}
\caption{Intersection number reduction -- when $B$ meets a puncture}
\label{fig:intnumreduction1}
\end{figure}

\begin{figure}
\includegraphics[height=0.11\textheight]{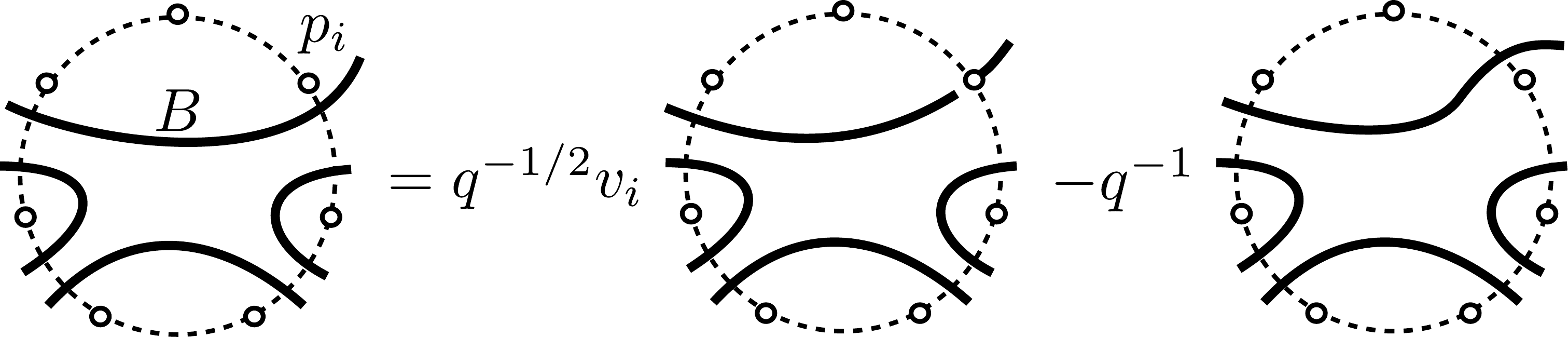}
\caption{Intersection number reduction -- when $B$ does not meet a puncture}
\label{fig:intnumreduction2}
\end{figure}

Now, it is sufficient to show that any skeins whose supports are in $\Sigma \setminus \mathrm{int}\; D$ is finitely generated. For this purpose, by stretching the boundary of $D$, we may draw $\Sigma \setminus \mathrm{int}\; D$ as in the right figure in Figure \ref{fig:handlebody}. This diagram is obtained by attaching several handles on a rectangular strip, and all marked points are on the strip itself, not on attached handles. In particular, there is no marked points on the handles. The long dashed boundary is $\partial D$, and the original boundary of $\Sigma$ is drawn on the top row. 

\begin{figure}
\includegraphics[height=0.12\textheight]{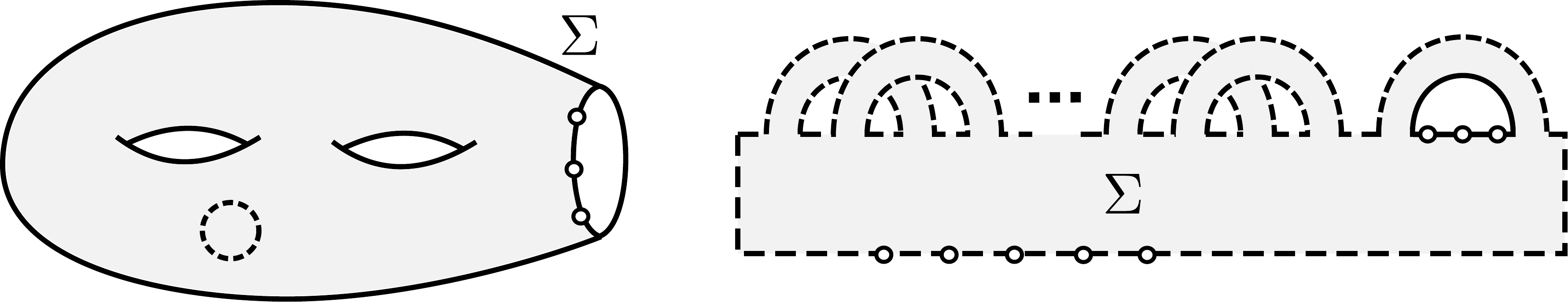}
\caption{Drawing of $\Sigma \setminus \mathrm{int}\; D$ -- the outer boundary is $\partial D$}
\label{fig:handlebody}
\end{figure}

Suppose that $\alpha \in \cS_{q}^{\MRY}(\Sigma)$ is a simple connected curve whose support is in $\Sigma \setminus \mathrm{int}\; D$. We fix an arbitrary orientation on $\alpha$. Assume that there is a handle $H$ on $\Sigma$ such that $\alpha$ traverses $H$ multiple times in the same direction. See the left figure in Figure \ref{fig:reduction1}. Then an ordinary skein relation can be applied to rewrite the curve in terms of curves that traverse $H$ strictly fewer times (Figure \ref{fig:reduction1}). Note that the rightmost diagram is a product of  two disjoint curve classes, and each of these curves traverses $H$ strictly fewer times than the original diagram on the left hand side of the equation. By applying this argument in pairs, we may assume that $\alpha$ travels each handle at most twice, and if so, it travels exactly twice in the opposite direction. 

\begin{figure}
\includegraphics[height=0.1\textheight]{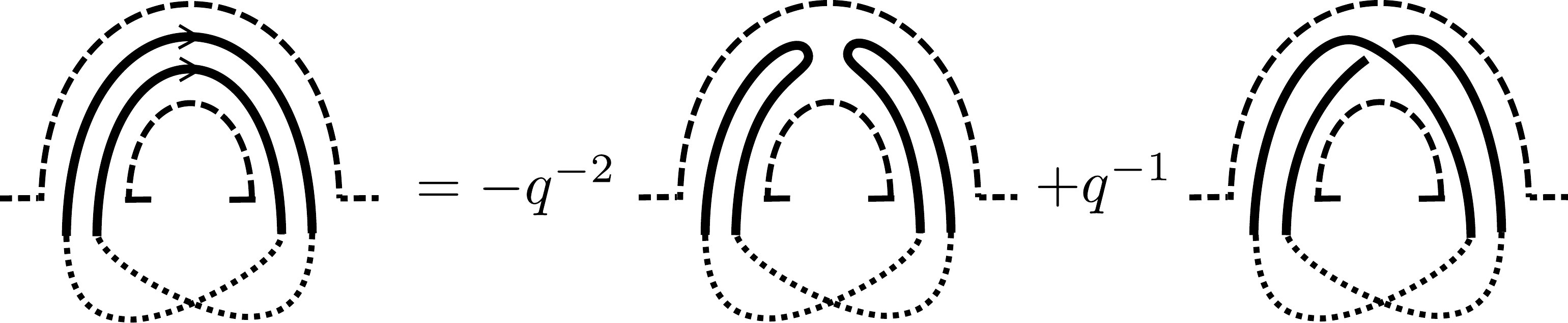}
\caption{Complexity reduction 1}
\label{fig:reduction1}
\end{figure}

Next, suppose that there is a handle $H$ where $\alpha$ travels twice, in opposite directions (top left of Figure \ref{fig:reduction2}). We may employ the ordinary skein relations three times, to write it as a combination of curves which traverse the handle at most once. See Figure \ref{fig:reduction2}. 

 Note that Bullock applied the above algorithm only to loops in \cite{Bul99}. However, the argument works equally well for arc classes. This is because the resolutions in Figures \ref{fig:reduction1} and \ref{fig:reduction2} do not change the part of the arc outside $H$ (on the parts of the arc diagram that are dotted in the figures). 

\begin{figure}
\includegraphics[height=0.31\textheight]{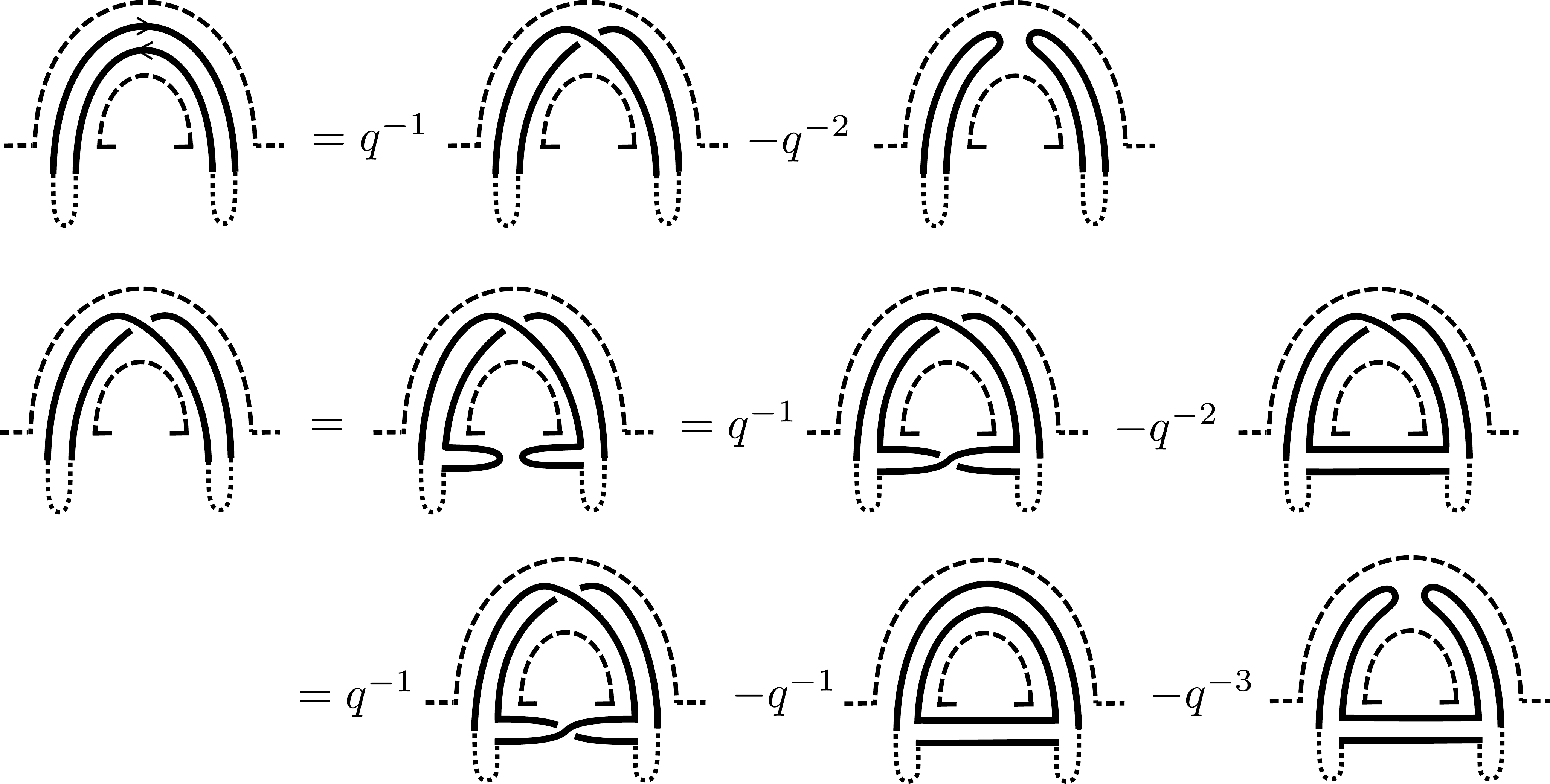}
\caption{Complexity reduction 2}
\label{fig:reduction2}
\end{figure}

In summary, any simple curve can be generated by simple curves that traverse each handle at most once.   These curves can be described by choosing the finite sequence of traveling handles in the case of loops and by choosing the initial/terminal marked points and the sequence of traveling handles in the case of arcs. (We do not claim that all sequences of handles will give simple curves and arcs -- they may induce self-intersections.) 
Thus, there are only finitely many possibilities. As a result, we obtain the proof of Theorem \ref{thm:finitegeneration}.

\begin{corollary}\label{cor:finitegenerationLRY}
The skein algebras $\cS_{q}^{\LRY}(\Sigma)$ and $\overline{\cS}_{q}^{\LRY}(\Sigma)$ are finitely generated. 
\end{corollary}

\begin{proof}
The skein algebra $\cS_{q}^{\LRY}(\Sigma)$ can be generated by 1) the image of a generating set of $\cS_q^\MRY(\Sigma)$ by $m : \cS_{q}^{\MRY}(\Sigma) \to \cS_q^\LRY(\Sigma)$, and 2) the same set of curves with different states.
The algebra $\cS_{q}^{\MRY}(\Sigma)$ is generated by finitely many loops and arcs, and for each arc connecting boundary marked points, there are only finitely many ways to impose the states. Therefore, $\cS_{q}^{\LRY}(\Sigma)$ is finitely generated, too. A homomorphic image of a finitely generated algebra is also finitely generated, so the same thing holds for $\overline{\cS}_{q}^{\LRY}(\Sigma)$. 
\end{proof}

\begin{remark}
One can show $\overline{\cS}_{q}^{\LRY}(\Sigma)$ is finitely generated in another way, directly from from Proposition \ref{prop:MRYvsLRY}. Since $\cS_{q}^{\MRY}(\Sigma)[\partial^{-1}]$ is an Ore localization of a finitely generated algebra by a finitely generated multiplicative subset $\partial$, it is also finitely generated. 
\end{remark}

\begin{remark}\label{rmk:minimalgenerator}
The finite generating set from the proof of Theorem \ref{thm:finitegeneration} is not guaranteed to be minimal. Note that the non-boundary handles appear as pairs. Suppose that $\alpha$ travels both handles (say $H_{1}$ and $H_{2}$) in a pair. Bullock proved that $\alpha$ can be generated by curves that passes $H_2$ immediately traveling after $H_1$ \cite{Bul99}. A minimal set of generators is not known, and is closely related to the computation of an explicit presentation of $\cS_q^\MRY (\Sigma)$. See Section \ref{ssec:presentation}. 
\end{remark}

If we remove each interior puncture by a new boundary component without a marked point, the same proof of Theorem \ref{thm:finitegeneration} works for $\cS_{q}^{\rM+}(\Sigma)$. Then we also obtain that  $\cS_{q}^{\rL+}(\Sigma)$ is finitely generated in the same way. Any homomorphic image of a finitely generated algebra is also finitely generated, so we obtain the following corollary. 

\begin{corollary}\label{cor:finitegeneration}
The skein algebras $\cS_{q}^{\rM+}(\Sigma)$, $\cS_{q}^{\rM}(\Sigma)$, $\cS_{q}^{\rL+}(\Sigma)$, $\cS_{q}^{\rL}(\Sigma)$, $\overline{\cS}_{q}^{\rL+}(\Sigma)$, and $\overline{\cS}_{q}^{\rL}(\Sigma)$ are all finitely generated.
\end{corollary}

\begin{remark}\label{rmk:algebrastructure}
For any oriented surface $\Sigma$ with $n$ interior punctures, we may draw the punctures on a small circular disk $D$ as we did in the proof of Theorem \ref{thm:finitegeneration}. Then for a small open neighborhood $D_n$ of $D$, there is a natural morphism 
\begin{equation}
	\cS_{q}^{\MRY}(D_n) \to \cS_{q}^{\MRY}(\Sigma). 
\end{equation}
In other words, $\cS_{q}^{\MRY}(\Sigma)$ admits an $\cS_{q}^{\MRY}(D_n)$-algebra structure. Note that $D_n$ is homeomorphic to $\RR^2_n$, the $n$-punctured $\RR^2$. Its skein algebra was calculated in \cite[Section 5]{ACDHM21}. 
\end{remark}

\section{Cluster algebra of surface} \label{sec:clusteralg}

We next consider the \emph{cluster algebra} $\cA(\Sigma)$ coming from the curves in a triangulable oriented surface $\Sigma$.  In this section, we give a brief definition of $\cA(\Sigma)$. A more detailed description and further properties of $\cA(\Sigma)$ can be found in \cite{FST08, FT18}.  In Section \ref{sec:compatibility}, we will explain its relationship with the skein algebra.

\subsection{Informal description}\label{ssec:informaldescription}

Let $\Sigma = (\underline{\Sigma}, V)$ be a surface that admits an ideal triangulation $\Delta$, with the set of vertices $V$. Here an ideal triangulation always contains all boundary arcs. Let $E := \{x_{i}\}$ be the set of edges in $\Delta$. The cluster algebra $\cA(\Sigma)$ is defined as some subalgebra of the field of fractions $\CC(E) := \CC(x_{i}|_{i \in I})$ generated by edges in $E$ subject to geometrically motivated relations.  Because the formal definition will require some additional structure on the edges, it may seem overly technical at first.  We thus start with an intuitive description to give the flavor of the construction and worry about the technical details later.

The basic construction of the cluster algebra is based on the following idea. For a given triangulation $\Delta$ and a choice of an internal diagonal edge $x$ of a quadrilateral, we may flip the diagonal edge $x$ in the triangulation to the other diagonal edge $x'$ of the quadrilateral.  This move produces a new triangulation $\Delta'$ where all edges are the same as in $\Delta$ except that $x$ is replaced by a flipped edge $x'$ (Figure \ref{fig:Ptolemy}). 

In the cluster algebra, we relate the edges $x$ and $x'$ using the \emph{Ptolemy relation}, which comes from the geometry of relating the lengths of the edges of the quadrilateral with its diagonals.   
More specifically, suppose that the quadrilateral has edges $x_{1}$, $x_{2}$, $x_{3}$, and $x_{4}$, and flipping the diagonal $x$ of the quadrilateral gives a new diagonal $x'$, as in Figure \ref{fig:Ptolemy}. In this case, the Ptolemy relation between $x$ and $x'$ is defined to be  $x x' = x_{2} x_{4} + x_{1}x_{3}$.   Observe that the product $ x x'$ is a `binomial' with respect to the quadrilateral edges $x_{1}$, $x_{2}$, $x_{3}$, and $x_{4}$, and that $x'$ is a rational function with respect to edges in $\Delta$.  As a result,  we can think of the edges in the new triangulation $\Delta'$ as belonging in the field of fractions of the edges of the original triangulation $\Delta$.  Moreover, the combinatorial changes between the edges of $\Delta$ and $\Delta'$ can also be encoded algebraically by way of mutation of the adjacency matrices, and we will provide these equations later.  

\begin{figure}
\includegraphics[height=0.13\textheight]{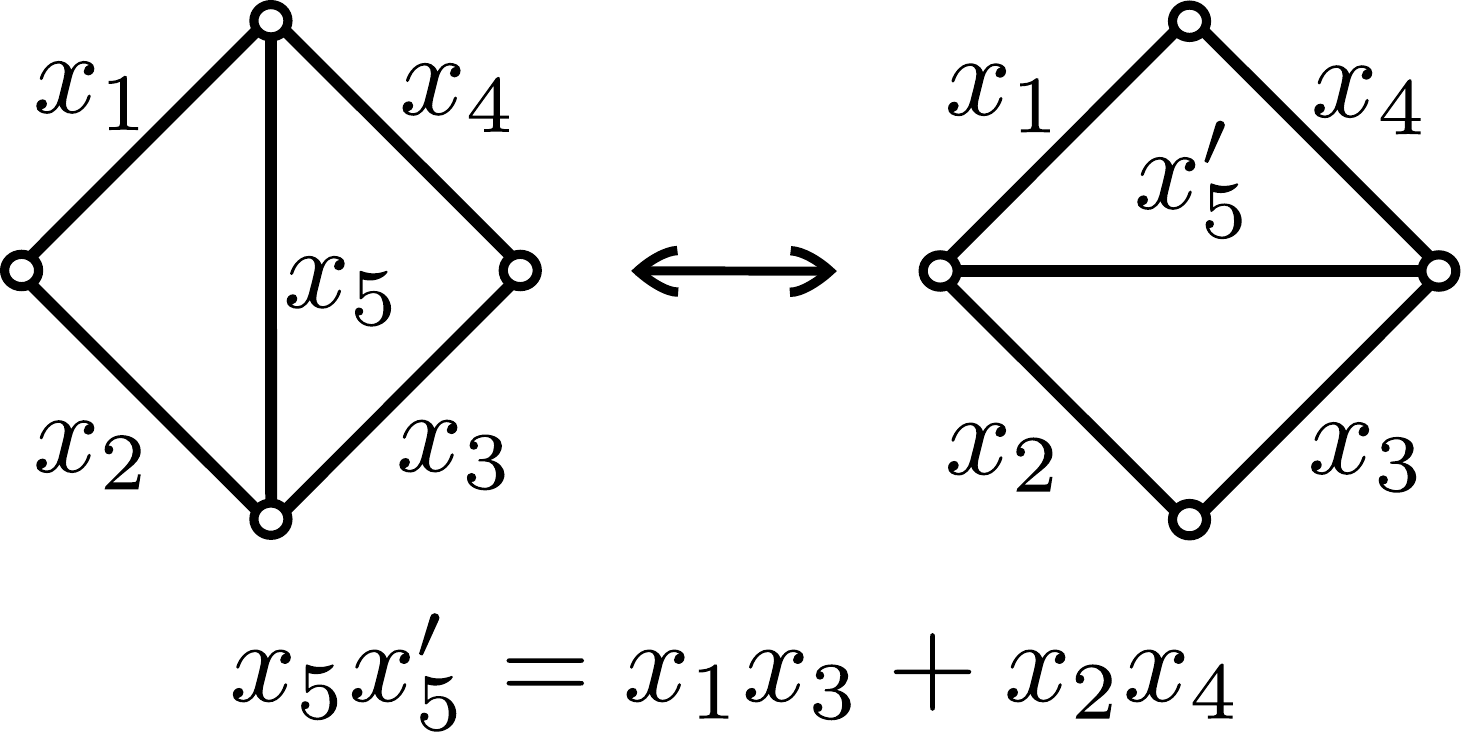}
\caption{Example of a flip of triangulation and the associated Ptolemy relation}
\label{fig:Ptolemy}
\end{figure}

The cluster algebra $\cA(\Sigma)$ we seek to define is very `close' to the subalgebra of $\CC(E)$ generated by all the edges of all possible triangulations subject to the Ptolemy relation and mutation of adjacency matrices.  Importantly, the resulting $\cA(\Sigma)$ can be regarded as an algebra of all arcs on $\Sigma$. Recall that, any embedded arc between two vertices belongs to some ideal triangulation on $\Sigma$. Moreover, any two ideal triangulations are connected by finitely many flips. In the algebra,  each time we flip an edge in $\Delta'$, we obtain a new triangulation with one new edge that is related to the old ones by the Ptolemy relation.  Although the number of generators of the cluster algebra is a priori infinite, it turns out that any arc can be written in terms of the edges of the original triangulation $\Delta$. The upshot is that we can think of any embedded arc of $\Sigma$ as belonging in $\cA(\Sigma)$, and the cluster algebra $\cA(\Sigma)$ is an algebra of curves.

However, there are some technical issues that arise when aligning this intuitive construction of $\cA(\Sigma)$ with the standard definition of cluster algebras.  In particular,  all elements of a cluster algebra should be flippable.  However, boundary edges are never diagonals of any quadrilaterial and hence never flippable.  This can be  accounted for by using a version of the cluster algebra where the boundary edges are \emph{frozen variables}. In cluster algebra literature, it is common to include the multiplicative inverses of the frozen variables, so we will take the same convention.  

A more serious issue comes up when there are interior punctures.  In particular, consider the part of an ideal triangulation in Figure \ref{fig:noflip}. One may flip the top interior edge, to get a new triangulation with a self-folded triangle. However, now the bottom edge connecting the two bottom punctures (inside of the self-folded triangle) is no longer a diagonal of a quadrilateral and thus cannot be flipped.  

\begin{figure}
\includegraphics[height=0.1\textheight]{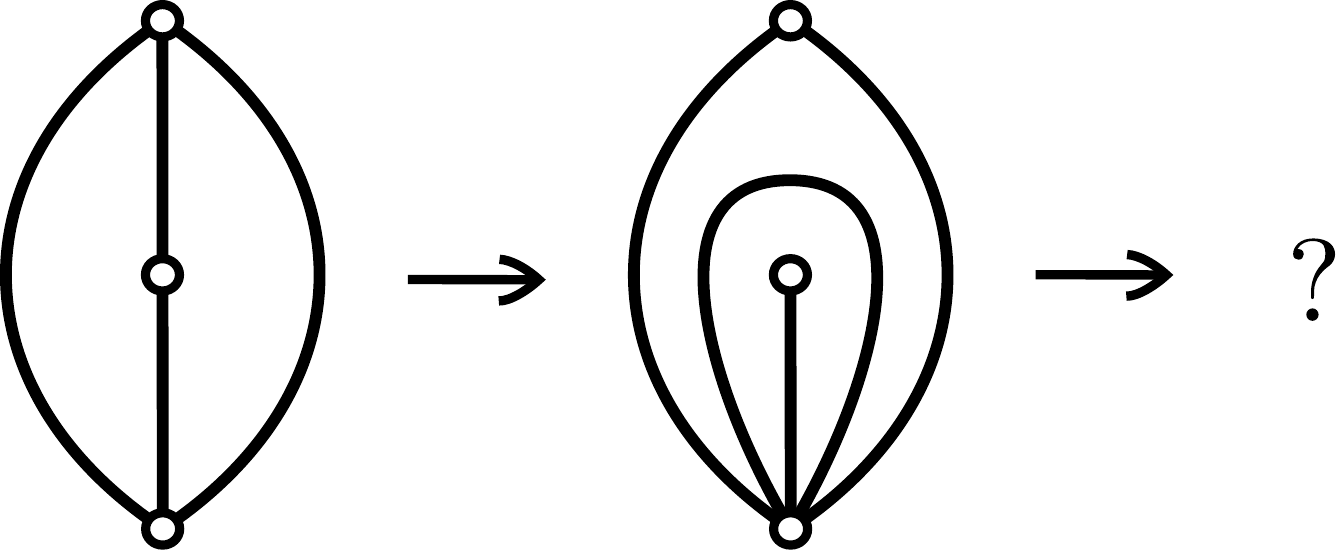}
\caption{No flip for self-folded triangles}
\label{fig:noflip}
\end{figure}

To rectify this issue, in \cite{FST08}, Fomin, Shapiro, and Thurston introduced the additional structure of \emph{tagged arcs} to account for the combinatorics of self-folded triangles.  A tagged arc is a topological arc where each end can be decorated by one of two taggings, plain or notched.  The notched end is denoted by drawing a small bowtie, and the plain end is unmarked. The three possible types of tagged arcs are illustrated in Figure \ref{fig:taggedarcs}.

\begin{figure}
\includegraphics[height=0.04\textheight]{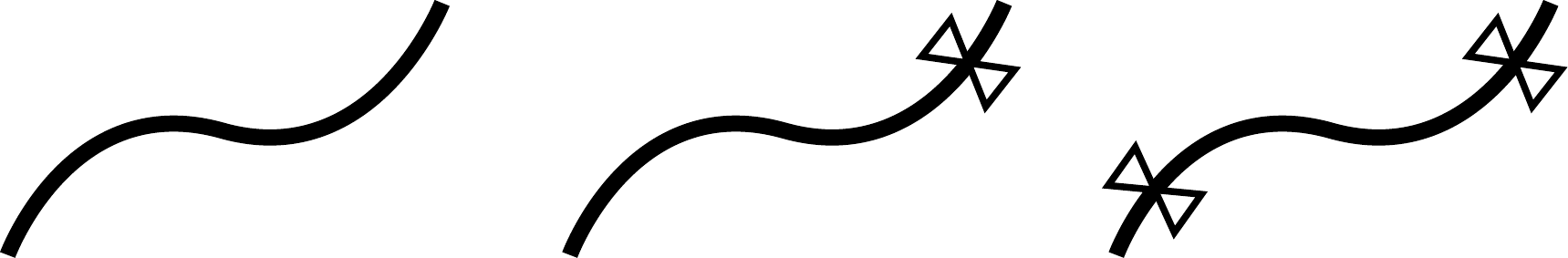}
\caption{Three types of tagged arcs}
\label{fig:taggedarcs}
\end{figure}

Triangulations can then be generalized to include tagged arcs, and every tagged arc can now be flipped as in Figure \ref{fig:taggedflip}.  This new diagrammatic procedure for tagged arcs is called a \emph{tagged flip}, and its combinatorics are designed to match that of ordinary flips along diagonals of quadrilaterals.   For the details, see \cite[Section 7]{FST08}, \cite[Section 3]{MW24}.

\begin{figure}
\includegraphics[height=0.1\textheight]{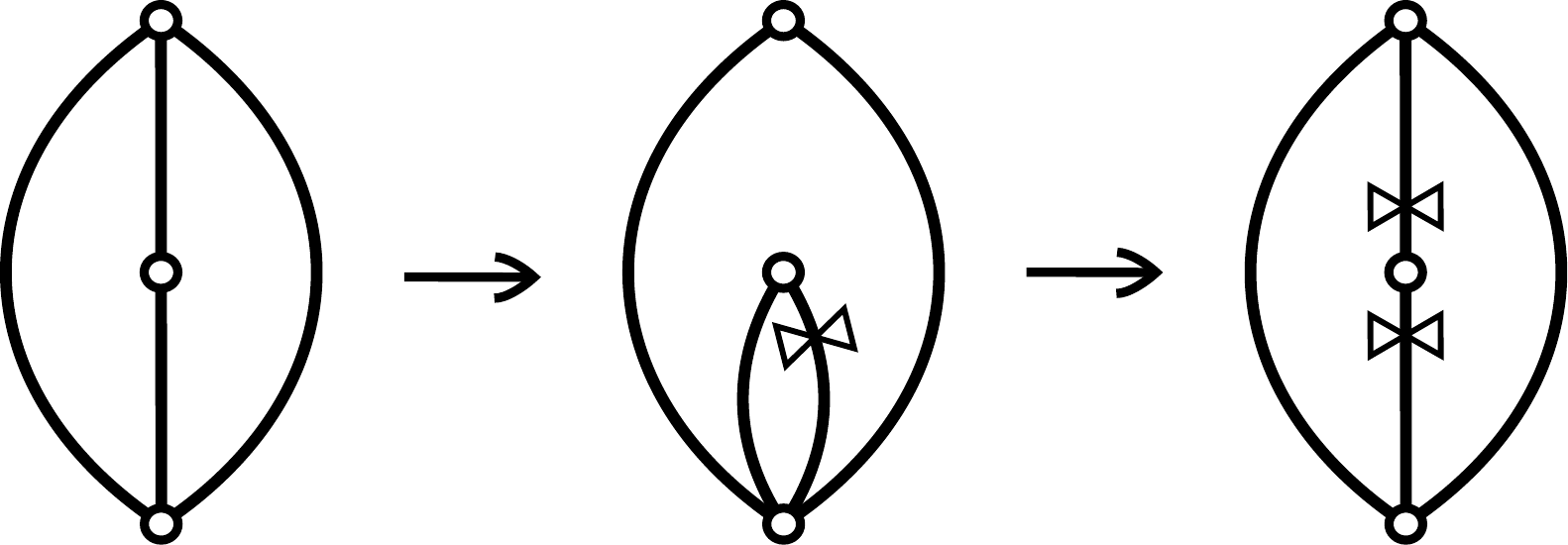}
\caption{Tagged flip}
\label{fig:taggedflip}
\end{figure}

We emphasize that a tagged arc can have a notched end only at an interior puncture. Boundary arcs and arcs connecting boundary marked points cannot have any tags. So, in the absence of interior punctures, the additional structure of tagging is not needed at all to define the cluster algebra of the surface.  

\subsection{Formal definition}

In this section, we give a quick reminder of the construction of the cluster algebra of surfaces. For the general theory of cluster algebras, \cite{Wil14} is an excellent introduction. For the details of the tagged arcs and tagged triangulations, see \cite[Section 7]{FST08}. 

Let $\Delta$ be a tagged ideal triangulation on $\Sigma$. In other words, $\Delta$ is a maximal collection of pairwise compatible tagged arcs.\footnote{Here, two tagged arcs are said to be compatible if their untagged versions are disjoint except at $V$ and their taggings satisfy a technical condition if they share an endpoint: if the two arcs are not isotopic and meet at an endpoint, then their taggings at that end must be identical;  and if the two arcs are isotopic and meet at an endpoint, then at least one one of the two ends must be identical.}  Let $E$ be the set of edges in $\Delta$, and $I$ be a fixed index set for edges in $E$. We divide $I =I^{\Delta} =  I_{\circ} \sqcup I_{\partial}$, where $I_{\circ}$ is the index set for interior edges and $I_{\partial}$ is the index set for boundary edges. 

A \emph{seed} is a collection of data $(E, B)$, where $E$ is the set of edges in $\Delta$ and $B = (b_{ij})$ is an $I \times I$ matrix such that its entry is given by Figure \ref{fig:adjacency}, that is, a signed counting of the adjacency of two edges $x_{i}$ and $x_{j}$ on the triangulation, i.e. $x_i$ and $x_j$ are edges of an ideal triangle. Thus, $-2 \le b_{ij} \le 2$ and $B$ is a skew-symmetric matrix. The matrix $B$ is called the \emph{exchange matrix}. 

\begin{figure}
\includegraphics[height=0.06\textheight]{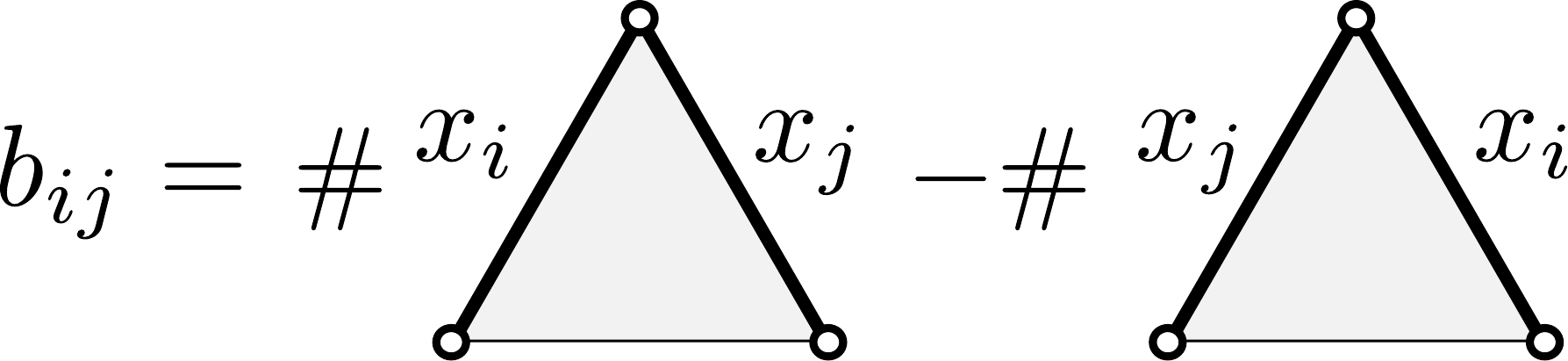}
\caption{Adjacency matrix $B$}
\label{fig:adjacency}
\end{figure}

For each choice of $k \in I_{\circ}$, we can construct a new seed in the following way.  Observe that there is a unique way to flip $x_{k}$ to another arc $x_{k}'$, corresponding to the procedure of Figure \ref{fig:Ptolemy} or \ref{fig:taggedflip}. 
We may thus replace $x_{k}$ by $x_{k}'$ in $\Delta$ to obtain a new tagged triangulation $\Delta'$. Let $E'$ denote the new set of edges in $\Delta'$. 
For edges $i \ne k$, $x_{i} = x_{i}'$.  In the cluster algebra, the \emph{exchange relation} relates the variables $x_{k}$ and $x_{k}'$ by 
\begin{equation}\label{eqn:exchangerelation}
	x_{k}x_{k}' = \prod_{b_{jk} > 0}x_{j}^{b_{jk}} + \prod_{b_{jk} < 0}x_{j}^{-b_{jk}}.
\end{equation}
The new exchange matrix $B' = (b_{ij}')$ for $\Delta'$ is then
\begin{equation}
	b_{ij}' := \begin{cases}
	-b_{ij}, & \mbox{ if } i = j \mbox{ or } j = k,\\
	b_{ij} + \frac{1}{2}(|b_{ik}|b_{kj} + b_{ik}|b_{kj}|), & \mbox{ otherwise}.
	\end{cases}
\end{equation}
The \emph{cluster mutation}  in the $k$-th direction is defined to be the map $\mu_{k} : (E, B) \to (E', B')$. One can check that this formal definition captures the combinatorial changes from flipping an edge. One can also check that $\mu_{k}$ is an involution, i.e., $\mu_{k}^{2} = \mathrm{id}$. 

Notice that boundary edges correspond to frozen variables.  In particular, cluster mutation is not defined for $k \in I_{\partial}$.

\begin{definition}\label{def:clusteralgebra}
Let $\Sigma$ be a triangulated surface with an initial seed $(E, B)$ coming from a fixed triangulation $\Delta$. The \emph{cluster algebra $\cA(\Sigma)$} is a subalgebra of $\CC(E)$ generated by 
\begin{equation}
\bigcup_{(E', B')}(\{x_{i}'|_{i \in I_{\circ}} \} \cup \{x_{j}'^{\pm1}|_{j \in I_{\partial}}\})
\end{equation}
for all cluster seeds obtained by taking a finite sequence of cluster mutations. 
\end{definition}

It is well-known that $\cA(\Sigma)$ does not depend on the choice of the initial seed. Thus, it is an invariant of a marked surface $\Sigma = (\underline{\Sigma}, V)$. 

\begin{example}\label{ex:oncepuncturedtorus}
Let $\Sigma_{1,1}$ be the torus with one interior puncture. We draw the torus using a standard square diagram as in Figure \ref{fig:torus}. We fix an ideal triangulation $\Delta$ with $E = \{x_1, x_2, x_3\}$. One may check that its exchange matrix is
\begin{equation}
    B = \left[\begin{array}{rrr}0&2&-2\\
    -2&0&2\\2&-2&0\end{array}\right].
\end{equation}
Flipping $x_3$, we obtain a new triangulation as in the right figure in Figure \ref{fig:torus}. This cluster mutation $\mu_3$ results in a new seed $(E' = \{x_1, x_2, x_3'\}, B')$, where 
\begin{equation}
    B' = \left[\begin{array}{rrr}0&-2&2\\
    2&0&-2\\-2&2&0\end{array}\right].
\end{equation}
\end{example}

\begin{figure}
\begin{tikzpicture}[scale=0.8]
	\draw[line width = 2pt] (0, 0) -- (0, 3);
	\draw[line width = 2pt] (0, 0) -- (3, 0);
	\draw[line width = 2pt] (3, 0) -- (3, 3);
	\draw[line width = 2pt] (0, 3) -- (3, 3);
    \draw[line width = 2pt] (0, 0) -- (3, 3);
    \node at (3.5, 1.5) {$x_1$};
    \node at (-0.5, 1.5) {$x_1$};
    \node at (1.5, -0.5) {$x_2$};
    \node at (1.5, 3.5) {$x_2$};
    \node at (1.7, 1) {$x_3$};
\end{tikzpicture}
\quad
\begin{tikzpicture}[scale=0.8]
	\draw[line width = 2pt] (0, 0) -- (0, 3);
	\draw[line width = 2pt] (0, 0) -- (3, 0);
	\draw[line width = 2pt] (3, 0) -- (3, 3);
	\draw[line width = 2pt] (0, 3) -- (3, 3);
    \draw[line width = 2pt] (3, 0) -- (0, 3);
    \node at (3.5, 1.5) {$x_1$};
    \node at (-0.5, 1.5) {$x_1$};
    \node at (1.5, -0.5) {$x_2$};
    \node at (1.5, 3.5) {$x_2$};
    \node at (1.3, 1) {$x_3'$};
\end{tikzpicture}
\caption{Once punctured torus}
\label{fig:torus}
\end{figure}
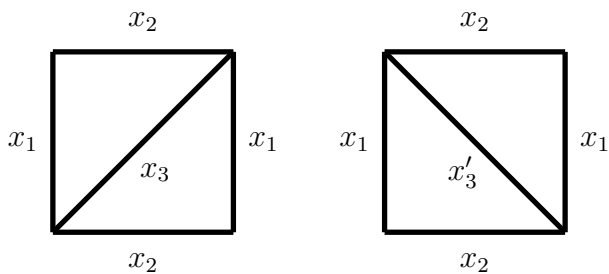

\begin{remark} 
When $\Sigma = \Sigma_{g, 1}$ is a surface with one interior puncture and without boundary, one can check that no triangulation admits any self-folded triangle as in the second Figure in Figure \ref{fig:noflip} and any arc is flippable. In particular, tagged arcs are unnecessary to define $\cA(\Sigma)$.
\end{remark}

\begin{theorem}[\protect{\cite[Theorem 7.11]{FST08}}]\label{thm:clusteralgebraofsurf}
Let $\Sigma$ be a triangulated surface with an initial seed $(E, B)$ coming from a fixed triangulation $\Delta$. Suppose that $\Sigma$ is not a once-punctured surface without boundary. Then the cluster algebra $\cA(\Sigma)$ is generated by arcs and tagged arcs. If $\Sigma$ is a once-punctured surface without boundary, or a surface without interior puncture, then $\cA(\Sigma)$ is generated by ordinary arc classes. 
\end{theorem}

\begin{remark}\label{rem:generalclusteralgebra}
In this paper, we focus on the cluster algebra $\cA(\Sigma)$ of a triangulated surface $\Sigma$. However, the cluster algebra $\cA$ can be defined for more general situation. Suppose that we have a finite list of variables $E = \{e_{i}\}_{i \in I}$, a decomposition of the index set $I = I_{unf} \sqcup I_{fro}$ (Note that in the case of $\cA(\Sigma)$, $I_{unf} = I_\circ$ and $I_{fro} = I_\partial$.) and a skew-symmetric integral matrix $B$. Then $\cA$ is defined as subalgebra of $\CC(\{e_{i}\})$ that is generated by 
\begin{equation}
\bigcup_{(E', B')}(\{x_{i}'|_{i \in I_{unf}}\} \cup \{ x_{j}'^{ \pm1}|_{j \in I_{fro}}\})
\end{equation}
where $(E', B')$ is a cluster seed obtained by taking a finite sequence of cluster mutations. 
\end{remark}

\subsection{Upper cluster algebra}\label{ssec:uppercluster}

For each cluster algebra $\cA$, one may define another algebra, the so-called upper cluster algebra $\cU$, which seems to be more natural from the algebraic geometry viewpoint and behave better algebraically \cite{BFZ05}. In this section, we review the definition of the upper cluster algebra $\cU(\Sigma)$ in the context of a cluster algebra of a surface $\cA(\Sigma)$. 

Let $\Sigma$ be a triangulated marked surface and $\cA(\Sigma)$ be the cluster algebra associated to $\Sigma$. Recall that any two cluster seeds $(E, B)$ and $(E', B')$ are connected by a finite sequence of cluster mutations. So by applying the exchange relation \eqref{eqn:exchangerelation} finitely many times, we may algebraically relate $E$ with $E'$.  More specifically, if we denote $E = \{x_{i}\}$ and $E' = \{x_{i}'\}$, any $x_{i}$ can be written as a rational function with respect to $\{x_{i}'\}$. One can say more about the type of rational functions obtained in this way. 

\begin{theorem}[Laurent phenomenon \protect{\cite[Theorem 3.1]{FZ02}}]\label{thm:Laurentphenomenon}
For any two cluster seeds $(E = \{x_{i}\}, B)$ and $(E', \{x_{i}'\}, B')$ for $\cA$, 
\begin{equation}
	x_{i}' \in \CC[x_{1}^{\pm1}, x_{2}^{\pm1}, \cdots, x_{n}^{\pm1}].
\end{equation}
In other words, $x_{i}'$ is a Laurent polynomial with respect to any cluster seed. 
\end{theorem}

\begin{definition}\label{def:upperclusteralgebra}	
Consider a cluster algebra $\cA \subset \CC(E)$. The \emph{upper cluster algebra} $\cU$ is defined as the set of all elements in $\CC(E)$ that can be written as a Laurent polynomial with respect to arbitrary seeds. In other words, 
\begin{equation}
	\cU := \bigcap_{(E' = \{x_{i}'\}, B')}\CC[x_{1}'^{ \pm1}, x_{2}'^{ \pm1}, \cdots, x_{n}'^{ \pm1}].
\end{equation}
\end{definition}

Now Theorem \ref{thm:Laurentphenomenon} implies that $\cA$ is a subalgebra of $\cU$. There are several cases that one can show $\cA = \cU$ \cite{Mul13, MSW13, CLS15, Mul16, MS16, IOS23}. However, in general, $\cU$ maybe strictly larger than $\cA$, and deciding whether $\cA = \cU$ or not is a challenging problem. See Section \ref{ssec:applications} for the case of $\cA(\Sigma)$.

\section{Compatibility between skein algebra and cluster algebra}\label{sec:compatibility}

We now prove Theorem \ref{thm:compatibility}, starting with a review of relevant definitions. 

\subsection{Cluster algebra in skein algebra}
Fix a tagged triangulation $\Delta$, or equivalently, a cluster seed $(E = \{x_{i}\}, B)$ of $\cA(\Sigma)$ where each $x_{i}$ corresponds to a tagged arc. Note that if a tagged arc $x_{i}$ ends at an interior puncture $v \in V_{\circ}$, then that end of $x_{i}$ can be decorated in one of two ways:  plain or notched. See Figure \ref{fig:taggedarcs}. For each tagged arc $\alpha$, let $\underline{\alpha}$ be its underlying topological arc after forgetting its tags. We can then understand $\underline{\alpha}$ as an element of $\cS_{1}^{\MRY}(\Sigma)$.

Consider the subalgebra $\CC[x_{i}|_{i \in I_{\circ}}, x_{j}^{\pm1}|_{j \in I_{\partial}}] \subset \CC(E)$.

\begin{definition}\label{def:maprho}
Fix a cluster seed $(E = \{x_{i}\}, B)$.  Suppose that $x_{i}$ is a tagged arc connecting $v$ and $w$ 
(not necessarily distinct). We set
\begin{equation}
	\rho_{\Delta}(x_{i}) := \begin{cases}
	\underline{x_{i}} & \mbox{ if both ends are tagged plainly,}\\
	v\underline{x_{i}} & \mbox{ if only the end at $v$ is notched,}\\
	w\underline{x_{i}} & \mbox{ if only the end at $w$ is notched,}\\
	vw\underline{x_{i}} & \mbox{ if both ends are notched.}
	\end{cases}. 
\end{equation}
By extend it using the universal property of the polynomial ring, we obtain an algebra homomorphism
\begin{equation}\label{eqn:localhom}
	\rho_{\Delta} : \CC[x_{i}|_{i \in I_{\circ}}, x_{j}^{\pm1}|_{j \in I_{\partial}}]
	\to \cS_{1}^{\MRY}(\Sigma)[\partial^{-1}].
\end{equation}
Note that any boundary edge $x_{j}$ is not tagged. Thus, $x_{j}^{\pm1}$ simply maps to $\underline{x_{j}}^{\pm1}$ and gives a well-defined element in $\cS_{1}^{\MRY}(\Sigma)[\partial^{-1}]$.
\end{definition}

The slogan for $\rho$ is: ``A tag is a vertex class.'' Whenever one has a notch on a tagged arc, replace the notched end by a vertex class. 

\begin{proof}[Proof of Part (1) of Theorem \ref{thm:compatibility}]
For each cluster seed, we have a well-defined homomorphism $\rho_{\Delta}$ in \eqref{eqn:localhom}. To prove the statement, we first show that these homomorphisms are compatible: If $\Delta$ and $\Delta'$ are connected by a single flip, then for any element in 
\begin{equation}
	\CC[x_{i}|_{i \in I_{\circ}}, x_{j}^{\pm1}|_{j \in I_{\partial}}]
	\cap \CC[x_{i}'|_{i \in I_{\circ}}, x_{j}'^{ \pm1}|_{j \in I_{\partial}}], 
\end{equation}
the image of $\rho_{\Delta}$ and $\rho_{\Delta'}$ coincide. Thus, we may `glue' the two homomorphisms to get $\rho$. 

One may describe all possible tagged flips and for each one, calculating the mutation explicitly. It was done in \cite[Proposition 4.3]{MW24} using puzzle pieces. The paper covered only surfaces without boundary, but note that we do not flip any boundary edge, so the same list of cases covers all possible cases here, too. 

To give a flavor of the type of computation, we now reproduce the proof using puzzle pieces for one particular case. Consider the tagged flip $\Delta \to \Delta'$ described in Figure \ref{fig:taggedflipexample}. For this flip, the only nontrivial change is made for $x$, and by the exchange relation \eqref{eqn:exchangerelation}, we obtain $xx' = x_{2} + x_{3}$. On the other hand, the puncture-skein relation in Definition \ref{def:MRY} (after setting $q = 1$) provides $v_2\underline{x}\underline{x'} = \underline{x_{2}}+\underline{x_{3}}$. Therefore, 
\begin{equation}
	\rho(xx') = v_2\underline{x}\underline{x'} = 
	\underline{x_{2}}+\underline{x_{3}} = \rho(x_{2}+x_{3}).
\end{equation}

\begin{figure}
\includegraphics[height=0.12\textheight]{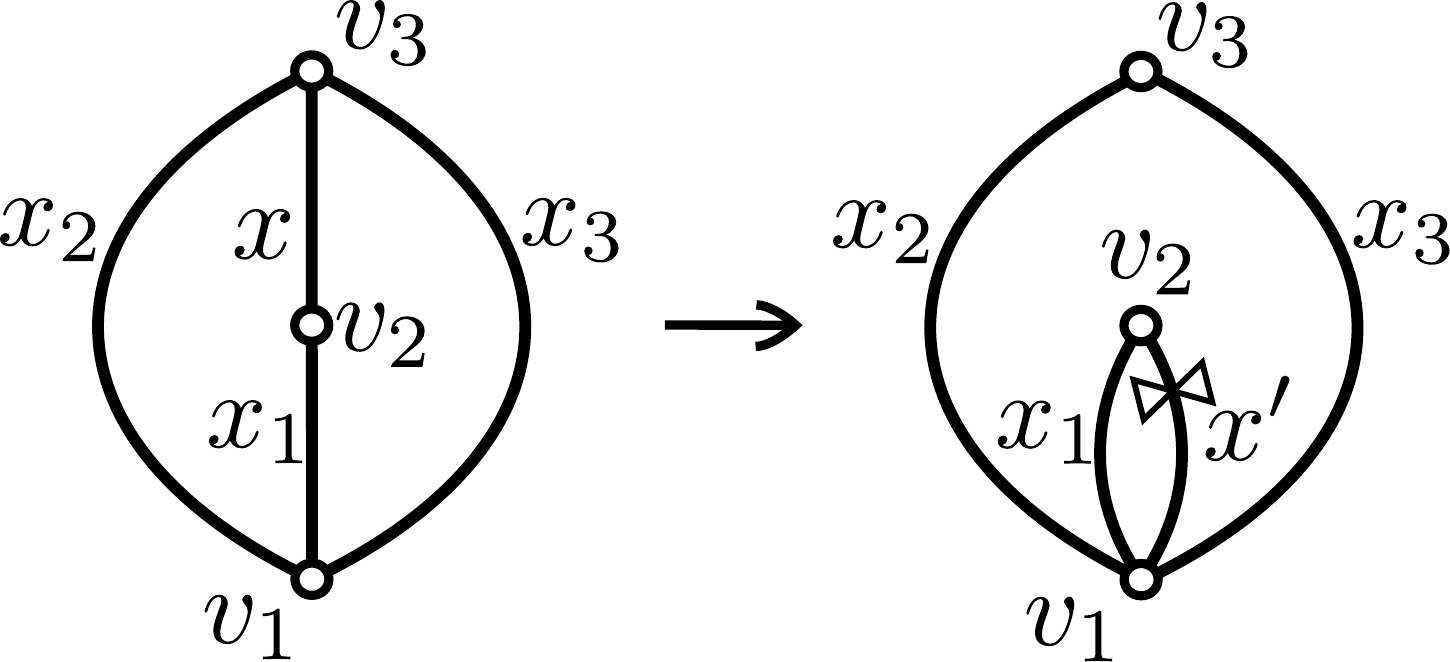}
\caption{A tagged flip of an edge of triangulation}
\label{fig:taggedflipexample}
\end{figure}

With similar calculations on other cases, we obtain a well-defined morphism $\rho : \cA(\Sigma) \to \cS_{1}^{\MRY}(\Sigma)[\partial^{-1}]$. To finish the proof, we need to show that $\rho$ is a monomorphism. 

Fix an ordinary triangulation $\Delta$ whose set of edges is $E = \{x_{i}\}$. We claim that both $\cA(\Sigma)$ and $\cS_{q}^{\MRY}(\Sigma)[\partial^{-1}]$ admit a morphism to the Laurent polynomials in $\{ x_{i}\}$, 
\[
	L _{\Delta}:= \CC[x_{1}^{\pm1}, x_{2}^{\pm1}, \cdots, x_{n}^{\pm1}].
\]
For $\cA(\Sigma)$, the Laurent phenomenon of Theorem \ref{thm:Laurentphenomenon} provides an inclusion $\cA(\Sigma) \subset L_\Delta$. We denote the inclusion map by $\iota$. For $\cS_{1}^{\MRY}(\Sigma)$, one needs to show that any vertices and loop classes are Laurent polynomials. The computation for the loop classes and general arc classes not necessarily in $\Delta$ is well-known in skein theory literature. For instance, see \cite[Theorem 3.22]{RY14}. Applying puncture-skein relation, any interior puncture is also a Laurent polynomial in $E$. For instance, in Figure \ref{fig:taggedflipexample}, $v_2\underline{x}\underline{x'} = \underline{x_{2}} + \underline{x_{3}}$. 

Thus, we have the following commutative diagram of commutative algebras:
\begin{equation}
	\xymatrix{\cA(\Sigma) \ar[rr]^{\rho} \ar[rd]^{\iota} &&\cS_{1}^{\MRY}(\Sigma)[\partial^{-1}]
	\ar[ld]_{\iota'}\\
	& L_{\Delta}.}
\end{equation}
Since $\iota = \iota' \circ \rho$ is injective, $\rho$ is injective, too. 
\end{proof}

\begin{remark}
Note that the above proof does not claim that $\iota'$ is also injective. We believe it has to be shown separately. If any edge is not a zero divisor, then the localization map $\iota'$ is injective. But how do we know that the multiplication of an arc on $\cS_{1}^{\MRY}(\Sigma)[\partial^{-1}]$ is injective? One way to show it is using dimension theory from algebraic geometry. Apply an algebraic ring extension $\widetilde{\cA}(\Sigma)$ of $\cA(\Sigma)$ in $\CC(E)$ and construct an epimorphism $\widetilde{\rho} : \widetilde{\cA}(\Sigma) \to \cS_{1}^{\MRY}(\Sigma)[\partial^{-1}]$. Since $\widetilde{\rho}$ is an epimorphism from an integral domain with the same dimension, its kernel is trivial \cite[Theorem 5.2]{MW24}. Thus, $\widetilde{\cA}(\Sigma) \cong \cS_{1}^{\MRY}(\Sigma)[\partial^{-1}]$ and the latter is an integral domain, and any nonzero element is not a zero divisor. 
\end{remark}

Observe that the above proof tells us that any loop class in $\cS_{1}^{\MRY}(\Sigma)[\partial^{-1}]$ is a Laurent polynomial with respect to any fixed ordinary triangulation. The same Laurent polynomial formula also applies to tagged triangulation \cite[Proposition 3.15]{MSW11}. It thus follows that all loop classes are in $\cU(\Sigma)$.  However, even though each vertex class is a Laurent polynomial with respect to an ordinary triangulation, it is not with respect to a tagged triangulation. Therefore, it does not belong to $\cU(\Sigma)$. 

\begin{definition}\label{def:squarealgebra}
Let $\cS_{q}^{\square}(\Sigma)$ be a subalgebra of $\cS_{q}^{\MRY}(\Sigma)[\partial^{-1}]$ generated by:
\begin{enumerate}
\item loop classes;
\item all arc classes;
\item all formal inverses of boundary arc classes;
\item if $\beta$ is an arc class connecting two (non-necessarily distinct) interior punctures $v$ and $w$; $v\beta$, $w\beta$, and $vw\beta$;
\item if $\beta$ is an arc class connecting an interior puncture $v$ and a boundary marked point, $v\beta$. 
\end{enumerate}
\end{definition}

\begin{proof}[Proof of Part (2) of Theorem \ref{thm:compatibility}]
When $q = 1$, we may define $\cS_{1}^{\square}(\Sigma)$ in a simpler way: It is a subalgebra of $\cS_{1}^{\MRY}(\Sigma)[\partial^{-1}]$ generated by the image of $\rho$ and the loop classes. Since $\cS_{1}^{\square}(\Sigma)$ is generated by $\rho(\cA(\Sigma))$ and the loop classes, any element in $\cS_{1}^{\square}(\Sigma)$ is in $\cU(\Sigma)$. 
\end{proof}

We also obtain the following structural results:

\begin{proposition}\label{prop:skeinstructure}
Let $\Sigma$ be a marked surface.
\begin{enumerate}
\item $\cS_{q}^{\square}(\Sigma)$ is finitely generated. 
\item $\cS_{1}^{\square}(\Sigma)$ is an integral domain. 
\end{enumerate}
\end{proposition}

\begin{proof}
One may note that if one start with a generator in $\cS_{q}^{\square}(\Sigma)$, the reduction process described in Section \ref{sec:finitegeneration} makes elements in $\cS_{q}^{\square}(\Sigma)$ only. Thus, the same proof of the finite generation works for $\cS_{q}^{\square}(\Sigma)$. This proves (1). 

Item (2) is an immediate consequence of Theorem \ref{thm:compatibility} because $\cU(\Sigma)$ is an integral domain. 
\end{proof}

\subsection{Applications of Compatibility}\label{ssec:applications}

There are a few applications to the structure of $\cA(\Sigma)$. 

In general, it has been observed that the ordinary cluster algebra $\cA$ may behave badly, but its upper cluster algebra has better structure. For instance, in many cases, $\cU$ is a finitely generated algebra, hence it defines an affine irreducible algebraic variety $\mathrm{Spec}\; \cU$. (However, there are some infinitely generated examples \cite{Spe13, GHK15}.) 
Moreover, $\cU$ is integrally closed \cite[Proposition 2.1]{Mul13}, so $\mathrm{Spec}\; \cU$ is a normal variety, which means its singularity type is relatively mild (the singular locus is of codimension $\ge 2$, etc.). So one may ask if $\cA = \cU$, for a given cluster algebra. 

In the case of cluster algebra $\cA(\Sigma)$ of surfaces, if there is at least one boundary component, then $\cA(\Sigma) = \cU(\Sigma)$ \cite{Mul13, MSW13, CLS15, MS16}. On the other hand, if $\Sigma$ is a surface without boundary but with one interior puncture, then $\cA(\Sigma) \ne \cU(\Sigma)$ \cite{Lad13}, which was shown by constructing an explicit element in $\cU(\Sigma) \setminus \cA(\Sigma)$.  

An application of the compatibility between $\cA(\Sigma_{g, n})$ and $ \cS_{1}^{\MRY}(\Sigma_{g, n})$ can be used to show that $\cA(\Sigma_{g, n}) \neq \cU(\Sigma_{g, n})$ when $g \geq 1$ \cite[Section 6]{MW24}. We summarize the argument here. Using techniques from invariant theory with $\ZZ/2\ZZ$-coefficients, it was shown that  $\cA(\Sigma_{g, n})$ with  $g,n \ge 1$ cannot be generated by finitely many elements \cite[Theorem C]{MW24}. By way of contradiction, assume $\cA(\Sigma_{g, n}) = \cU(\Sigma_{g, n})$. Since $\cA(\Sigma_{g, n}) \subset \cS_1^\square(\Sigma_{g, n}) \subset \cU(\Sigma_{g, n})$, this would imply that $\cA(\Sigma_{g, n}) = \cS_{1}^{\square}(\Sigma_{g, n})$. But $\cS_{1}^{\square}(\Sigma_{g, n})$ is finitely generated (Proposition \ref{prop:skeinstructure}), a contradiction. 

In the case of $\Sigma_{0, n}$, the $n$-punctured sphere, it was a folklore conjecture that $\cA(\Sigma_{0, n})$ is finitely generated. The computation in \cite{ACDHM21} indeed implies that $\cA(\Sigma_{0,n}) = \cS_{1}^{\square}(\Sigma_{0, n})$. Therefore, the finite generation of $\cA(\Sigma_{0, n})$ follows. We do not know if $\cA(\Sigma_{0, n}) = \cU(\Sigma_{0,n})$ or not. 

The finite generation problem for $\cA(\Sigma)$ itself is an interesting problem. As we mentioned earlier, it was shown that $\cA(\Sigma_{g, n})$ is not finitely generated when $g, n \ge 1$ \cite[Section 6]{MW24}. 
Note that from Proposition \ref{prop:skeinstructure} and Theorem \ref{thm:compatibility}, when $\cA(\Sigma) = \cU(\Sigma)$, we immediately obtain the finite generation of $\cA(\Sigma)$. The previous results and the discussion above show that if the given triangulated surface $\Sigma$ has boundary or $\Sigma = \Sigma_{0,n}$, $\cA(\Sigma)$ is finitely generated. Therefore, we have a complete understanding on the finite generation problem for $\cA(\Sigma)$.

\begin{remark}  
The skein algebra $\cS_{q}(\Sigma)$ and its variations are non-commutative (except for some small surfaces), and one can think of them as `quantum' versions of the commutative $\cS_{1}(\Sigma)$.   Similarly, one may wonder if we can obtain a `quantum' analogue $\cA_{q}(\Sigma)$ of $\cA(\Sigma)$. More generally, for a cluster algebra $\cA$ whose exchange matrix is of full-rank, its deformation quantization $\cA_{q}(\Sigma)$ was studied in \cite{BZ05} and called the quantum cluster algebra. However, if the exchange matrix is not of full-rank, it is not possible to obtain $\cA_{q}(\Sigma)$ \cite[Proposition 3.3]{BZ05}. It is also well-known that the exchange matrix $B$ for $\Sigma$ is not of full-rank if $\Sigma$ has an interior puncture. Thus, for a surface with an interior puncture, $\cA_{q}(\Sigma)$ does not provide a desired deformation quantization. 

On the other hand, note that $\cS_{q}^{\MRY}(\Sigma)$ can be defined for any oriented surface $\Sigma$ and it is a quantization of $\cS_{1}^{\MRY}(\Sigma)$. Therefore, from Theorem \ref{thm:compatibility}, we may understand $\cS_{q}^{\MRY}(\Sigma)$ as a quantization of $\cA(\Sigma)$. 
\end{remark}

\section{Open questions} \label{sec:open}

In this last section, we leave a few open questions on the structures of cluster algebras and skein algebras of surfaces.

\subsection{Algebraic structure of skein algebras}\label{ssec:presentation}

By Theorem  \ref{thm:finitegeneration},  $\cS_{q}^{\MRY}(\Sigma)[\partial^{-1}]$ and its variations are all finitely generated algebras.  However, numerous questions remain about its multiplicative structure.  For example, a presentation of the skein algebra is currently known only for a few small surfaces.  

\begin{question}
\begin{enumerate}
\item Compute a presentation of $\cS_{q}^{\MRY}(\Sigma)[\partial^{-1}]$. 
\item Find a minimal number of generators of $\cS_q^\MRY(\Sigma)[\partial^{-1}]$.
\end{enumerate}
\end{question}

 To the best of our knowledge, the presentation has been computed for the once-puncture torus $\Sigma_{1, 1}$ \cite{BPKW16} and for the $n$-punctured sphere $\Sigma_{0, n}$ \cite{ACDHM21}. Note that in these two cases, there is no boundary component, so the localization is not necessary. Even though it is not explicitly stated, the case of a disk with $n$ boundary marked points can be obtained from \cite[Section 12]{FZ03}. The case of the annulus with one marked point on each boundary component appears in \cite{Le15}.
 
Recall that $D_n$ is an $n$-punctured open disk. Note that when $q = 1$, the commutative algebra $\cS_{1}^{\MRY}(D_{n})$ naturally appears in many different contexts, including in classical invariant theory, as the coordinate ring of the Grassmannian, and in tropical geometry. Consult \cite[Remark 5.2]{ACDHM21} for additional discussion.

\begin{question}
Identify $\cS_q^\MRY(\Sigma)$ for a simple surface $\Sigma$ with other algebraic structures. 
\end{question}

One may investigate other algebraic properties of $\cS_q^\MRY(\Sigma)$. It has been known that $\cS_q^\MRY(\Sigma)$ is a domain \cite{BKL24} as a corollary of an embedding into a quantum torus, and its center has been calculated when $q$ is a primitive root of unity of odd order \cite{KMW25}. A similar result was obtained for $Z(\cS_q^{L+}(\Sigma))$ in \cite{Yu23a} and for $Z(\overline{\cS}_q^{L+}(\Sigma))$ in \cite{Kor21}. From its connection to the cluster algebra, it has been conjectured that $\cS_q^\MRY(\Sigma)$ has a `positive basis,' that is, a $\mathbb{Z}[q^{\pm1/2}][v_{i}^{\pm1}]$-basis whose multiplicative structure constants are all in $\mathbb{Z}_{\geq0}[q^{\pm1/2}][v_{i}^{\pm1}]$ \cite{Thu14} after we replace $\CC$ with $\mathbb{Z}[q^{\pm1/2}]$ in the ground ring. See related results in \cite{DM21, MQ23}.

\begin{question}
\begin{enumerate}
\item Do the generalized skein algebras each admit a positive basis? 
\item Compute the structure constants and describe them geometrically. 
\end{enumerate}
\end{question}

For a partial evidence for ${\cS}_1^\mathrm{RY} (\Sigma)$, see \cite{Kar24}. 

\subsection{Upper cluster algebra}
From Section \ref{sec:compatibility}, we have that $\cA(\Sigma) \subset \cS_{1}^{\square}(\Sigma) \subset \cU(\Sigma)$.  
Currently, there are no known elements in $\cU(\Sigma)$ but not in $ \cS_{1}^{\square}(\Sigma)$. This suggests:

\begin{question}
Is it true that $\cS_{1}^{\square}(\Sigma) = \cU(\Sigma)$?
\end{question}

Between $\cA(\Sigma)$ and $\cU(\Sigma)$, there is another algebra motivated by mirror symmetry, so-called the \emph{mid cluster algebra} $\cM(\Sigma)$ \cite{GHKK18}. It was shown that $\cS_{1}^{\square}(\Sigma) = \cM(\Sigma)$  \cite[Theorem 1.3]{MQ23}. It is not known whether or not $\cM(\Sigma)$ and $ \cU(\Sigma)$ are equal. 

Observe that if   $\cS_{1}^{\square}(\Sigma) = \cU(\Sigma)$, Proposition \ref{prop:skeinstructure} implies that $\cU(\Sigma)$ is finitely generated. 

This compatibility of $\cA(\Sigma)$ and $\cS_q^\MRY(\Sigma)$ seems to be a part of more general phenomenon. One may understand our $\cA(\Sigma)$ as the theory with the structure group $\mathrm{SL}_2$, in the framework of higher Teich\"muller theory in \cite{FG06}. For some other structure groups of row rank, there are similar results \cite{IOS23, IY23}. It suggests that for a higher structure group, where the skein theory is not clear, the skein algebra is expected to be `defined' as the upper cluster algebra.

\subsection{Representation theory of skein algebra}

A natural next step would be to study the representation theory of $\cS_{q}^{\MRY}(\Sigma)$.  There have been some investigations for quantum cluster algebras and localized Muller skein algebras with boundary edges  \cite{MNTY24, Kor21, Kor22, KK23}.  However, less is known about the representation theory for variations of the skein algebra in the presence of interior punctures, as in the Roger-Yang skein algebra.  

Based on techniques adapted from these earlier works,  \cite[Proposition 5.5]{KMW25} showed that when $q$ is an odd primitive root of unity, $\cS_{q}^{\MRY}(\Sigma)$ (hence $\cS_{q}^{\MRY}(\Sigma)[\partial^{-1}]$ too) is an \emph{almost Azumaya algebra}. This means `most' of irreducible representations can be identified with points in an open dense subset $U \subset \mathrm{MaxSpec}\; Z(\cS_{q}^{\MRY}(\Sigma))$. See \cite[Section 5]{KMW25} for more details. The open set $U$ is called the \emph{Azumaya locus} of $\mathrm{MaxSpec}\; Z(\cS_{q}^{\MRY}(\Sigma))$. The center $Z(\cS_{q}^{\MRY}(\Sigma))$ was calculated in \cite[Theorem A]{KMW25}. However, we do not yet have an explicit description of the representations of $\cS_{q}^{\MRY}(\Sigma)$.

\begin{question}
\begin{enumerate}
\item Describe the Azumaya locus $U \subset \mathrm{MaxSpec}\; Z(\cS_{q}^{\MRY}(\Sigma))$ explicitly. 
\item For each point $m \in U$, construct the corresponding finite dimensional irreducible representation $V$ of $\cS_{q}^{\MRY}(\Sigma)$ geometrically. 
\end{enumerate}

\end{question}

The ordinary skein algebra $\cS_q(\Sigma)$ can be regarded as a quantization of the Teichm\"uller space of $\Sigma$, and its representations are closely related to hyperbolic geometric data of the surface $\Sigma$. Similarly, since $\cS_{q}^{\MRY}(\Sigma)$ is based on combinatorial data from the decorated Teichm\"uller space of $\Sigma$ by extending the discussion in \cite{RY14, Mul16}, it would be interesting to see how the representations of $\cS_{q}^{\MRY}(\Sigma)$ are related to the hyperbolic geometry of $\Sigma$ as well.


\end{document}